\documentclass[12pt,letterpaper]{article}
\usepackage[left=1.4in,right=1.4in,top=1.5in,bottom=1in]{geometry}

\usepackage[utf8]{inputenc}
\usepackage{amsmath}
\usepackage{amsthm}
\usepackage{amsfonts}
\usepackage{amssymb}
\usepackage{mathtools}
\usepackage{hyperref}
\usepackage{algorithm}% http://ctan.org/pkg/algorithms
\usepackage[noend]{algpseudocode}% http://ctan.org/pkg/algorithmicx
\usepackage{tikz}
\usepackage[capitalize,noabbrev]{cleveref}
\usepackage{mathrsfs} 
\usepackage[title]{appendix}

\usepackage{color}

\usepackage{authblk}

\usepackage{graphicx}

{\theoremstyle{plain}

\newtheorem{lemma}{Lemma}
\newtheorem{theorem}{Theorem}

\newtheorem{example}{Example}
\newtheorem{assumption}{Assumption}
\newtheorem{definition}{Definition}
\newtheorem{conjecture}{Conjecture}
}
\def\Z{{\mathbb Z}}
\def\R{{\mathbb R}}
\def\N{{\mathbb N}}

\newcommand{\NP}{$\mathcal{NP}$}

\newcommand{\conv}{\operatorname{conv}}

\newcommand{\st}{\text{s.t.}}

\newcommand{\supp}{\operatorname{supp}}

\newcommand{\pare}[1]{\left(#1\right)}
\newcommand{\bra}[1]{\left\{#1\right\}}

\newcounter{claim}

 \usepackage[usenames,dvipsnames]{pstricks}
 \usepackage{epsfig}
 \usepackage{pst-grad} % For gradients
 \usepackage{pst-plot} % For axes

\makeatletter
\newcommand{\setwindow}[5]{
\def\xmin{#1}%
\def\ymin{#2}%
\def\xmax{#3}%
\def\ymax{#4}%
\pstFPsub\viewingwidth{#3}{#1}%
\pstFPdiv\result{\strip@pt#5}{\viewingwidth}%
\psset{unit=\result pt}}
\makeatother

\author[1]{Alberto Del Pia\thanks{delpia@wisc.edu}}
\author[2]{Jeff Linderoth\thanks{linderoth@wisc.edu}}
\author[3]{Haoran Zhu\thanks{hzhu94@wisc.edu, corresponding author}} 

\affil[1,2,3]{\small Department of Industrial and Systems Engineering, University of Wisconsin-Madison}
\affil[1,2]{\small Wisconsin Institute for Discovery, University of Wisconsin-Madison}

\title{New Classes of Facets for Complementarity Knapsack Problems}
\date{}

\begin{document}
\maketitle

\begin{abstract}
The \emph{complementarity knapsack problem} (CKP) is a knapsack problem with real-valued variables and complementarity conditions between pairs of variables. 
We extend the polyhedral studies of de Farias et~al.~for CKP, by proposing three new families of cutting-planes that are obtained from a combinatorial concept known as \emph{pack}. 
Sufficient conditions for these inequalities to be facet-defining, based on the new concept of a \emph{maximal switching pack}, are also provided.
Moreover, we answer positively a conjecture by de Farias et~al.~about the separation complexity of the inequalities introduced in their work, and propose efficient separation algorithms for our newly defined cutting-planes. 
\end{abstract}


\section{Introduction}

% One of the most pervasive problems in operations research are \emph{linear programs with complementarity constraints} (LPCC).
% In these problems, we are given matrices $A \in \R^{p \times m}, B, C \in \R^{p \times n}$ and $ d \in \R^p$, and we seek to optimize a linear function over the set:
% $$
% S^\perp= \left\{(x,y,z) \in \R^m \times \R^n \times \R^n \mid Ax + By + Cz \geq d, \ y_i \cdot z_i = 0 \ \forall i=1,\ldots, n \right\}.
% $$
% LPCCs have numerous applications, including multilevel optimization, inverse convex quadratic programming, indefinite quadratic programming, and piecewise linear optimization \cite{hu2012linear}.

% The first cutting-planes for LPCC were given in \cite{ibaraki1973use}, where Ibaraki used an earlier result of Balas \cite{MR290793} to obtain a family of cutting-planes from the simplex tableau of the linear programming relaxation.
% An improvement to these cuts was proposed by Sherali et al.~\cite{MR1397652} using disjunctive programming techniques. 
% Another early result on the inequality description of the convex hull of $S^\perp$ was given by Jeroslow \cite{jeroslow1978cutting}, who gave a finite sequential procedure for generating all valid inequalities of the convex hull of $S^\perp$, when every variable appears in some complementarity constraint.

In this paper, we investigate cutting-planes for the \emph{complementarity knapsack problem} (CKP).
Let $M = \{1, \ldots, m\}$ and, for every $i \in M,$ let $N_i = \{1, \ldots, n_i\}$, for some $m, n_i \in \N$. Let $b \in \R_+$ and $c_{ij}, a_{ij} \in \R_+$ for every $i \in M, j \in N_i$.
Then the \emph{CKP} is
\begin{align}
\max \quad & \sum_{i \in M} \sum_{j \in N_i} c_{ij} x_{ij} \label{CKP} \tag{CKP} \\
\st \quad & \sum_{i \in M} \sum_{j \in N_i} a_{ij} x_{ij} \leq b, \label{constraint: knapsack} \\
& x_{ij} \cdot x_{ij'} = 0, \qquad j \neq j' \in N_i, i \in M, \label{constraint: complementarity} \\
& 0 \leq x_{ij} \leq 1, \qquad j \in N_i, i \in M. \nonumber
\end{align} 
% Notice that, CKP restricts the linear constraint in LPCC to a knapsack constraint~\eqref{constraint: knapsack}, while it  generalizes the pairwise complementarity constraints in LPCC to allow multiple variables to be complementary with each other.
The constraints in \eqref{constraint: complementarity}, enforcing that at most one of the variables in the set $N_i$ takes a positive value are usually referred to as \emph{complementarity constraints}. 
Often, in the literature, the set of variables $\{x_{ij}\}_{j \in N_i}$, for any $i \in M,$ is called a \emph{special ordered set of type 1} (SOS1). 
When the variables are binary, the corresponding complementarity constraint is often called a \emph{generalized upper bound} (GUB).

Traditionally, \eqref{CKP} is modeled by introducing binary variables $y_{ij}$ for any $i \in M, j \in N_i$, and the complementarity constraints~\eqref{constraint: complementarity} are formulated by $x_{ij} \leq y_{ij}$ and $\sum_{j \in N_i} y_{ij} \leq 1$ for any $i \in M$. 
Then, cutting-planes are added into this mixed-integer programming (MIP) formulation in order to obtain a tighter relaxation, specifically cutting-planes from the knapsack polytope, see, e.g., \cite{balas1975facets,MR4259102,del2022multi,padberg19801,weismantel19970}.  
Note that all cutting-planes that are added into such a formulation are ``general purpose'', in the sense that their derivation does not take into consideration the structure of the specific problem at hand.
This approach naturally has several computational disadvantages, including the  size increase of the problem and the loss of structure, see \cite{de2001branch,fischer.pfetsch:18}. 
For these reasons, de Farias et al.~\cite{de2002facets} conducted a polyhedral study of CKP in the space of the continuous variables.
In particular, the authors presented two families of facet-defining inequalities, called \emph{fundamental complementarity inequalities}, that can be derived by lifting \emph{cover inequalities} of CKP. \footnote{These are not the typical ``cover inequalities'' in 0-1 programming} 
In a more recent paper, de Farias et al.~\cite{de2014branch} first generalized the two families of inequalities introduced in \cite{de2002facets}, and then numerically tested the performance of these inequalities by using them within a branch-and-cut algorithm. 
Our paper is motivated by these papers of de~Farias et al.~\cite{de2002facets, de2014branch}.
In particular, we further investigate the polyhedral structure of CKP in the space of the original variables. 
We remark that, the idea of dispensing with the use of auxiliary binary variables to model complementarity constraints and enforcing those constraints directly appears in a number of papers, including \cite{beale1970special, beaumont1990algorithm, keha2006branch, de2013polyhedral}.

In a \emph{branch-and-cut} procedure using the inequalities proposed in \cite{de2014branch} to solve CKP, a crucial step is \emph{separation}---given an optimal solution of the current linear relaxation, how can we efficiently identify an inequality in a given class to separate the given solution and strengthen the relaxation? 
de~Farias et al.~conjectured that the separation of their inequalities is \NP-complete. 
%Besides providing a simple separation heuristic for the inequalities  introduced, they further conjecture that their separation is \NP-complete. 
In our first contribution of this paper, we show that it is indeed \NP-complete to separate both classes of inequalities proposed in \cite{de2014branch}. 
Our second contribution is
%, in line with the main results in \cite{de2002facets, de2014branch}, 
the introduction of novel valid inequalities for the convex hull of the feasible set of \eqref{CKP}, which is coined the \emph{complementarity knapsack polytope} by de Farias et al.
We also provide sufficient conditions for these inequalities to be facet-defining. 
A key difference between our inequalities and the ones proposed in \cite{de2002facets, de2014branch}, is that the inequalities that we introduce cannot be derived via the lifting procedure proposed in \cite{de2002facets}. This is further discussed at the end of \cref{sec: cuts}.

This paper is organized as follows. 
In \cref{sec: recap}, we introduce some notation, assumptions, and previous results from the literature. 
In \cref{sec: complexity}, we provide the positive answers to the conjectures made by de~Farias et al.~\cite{de2014branch}. 
In \cref{sec: cuts}, we introduce the concept of \emph{pack} and \emph{maximal switching pack}, and propose three new families of cutting-planes for \eqref{CKP} which originate from packs.
We also show that these inequalities are facet-defining if the pack itself, or some particular subset of the pack, is a maximal switching pack. In \cref{sec: heuristics}, we propose efficient separation algorithms for the proposed cutting-planes.
These algorithms were not present in the conference version of this paper \cite{10.1007/978-3-031-18530-4_1}, where they were posed as interesting open questions.
In \cref{sec: future}, we discuss directions for future research.

\section{Notations, assumptions, and previous work}
\label{sec: recap}

To the best of our knowledge, the papers \cite{de2002facets,de2014branch} are the only polyhedral studies of the complementarity knapsack polytope. 
Using the same notation as in \cite{de2002facets}, we denote by $S$ the set of feasible solutions of \eqref{CKP}, and by $PS: = \conv(S)$ the \emph{complementarity knapsack polytope}. 
We denote by $d$ the number of variables in the problem, i.e., $d := \sum_{i \in M} n_i$. 
For ease of notation, we denote by $ij$ the ordered pair $(i,j)$.
We also identify any set containing exactly one element with the element itself. 
We denote by $I: = \cup_{i \in M} (i \times N_i)$, which is simply the set of all indices of $x$. 
We denote by $M_0$ the set of indices in $M$ that do not lie in any complementarity constraint, i.e., $M_0: = \{i \in M \mid n_i = 1\}$.
For $T \subseteq I,$ we define $M_T: = \{i \in M \mid ij \in T \text{ for some } j \in N_i\}$.
For any $n \in \N,$ we let $[n]: = \{1,\ldots, n\}$. 
 For any vector $v \in \R^n,$ we denote by $\supp(v)$ the support set of vector $v$, i.e., $\supp(v) := \{i \in [n] \mid v_i  \neq 0\}$.

As in \cite{de2014branch}, we make the following assumptions:
\begin{assumption}[Assumption 1 in  \cite{de2014branch}]
$M \neq M_0$.
\end{assumption}

\begin{assumption}[Assumption 2 in  \cite{de2014branch}]
$\sum_{i \in M} \max\{a_{i1}, \ldots, a_{in_i}\} > b$.
\end{assumption}

\begin{assumption}[Assumption 3 in  \cite{de2014branch}]
\label{asmp: 3}
$b > 0, a_{ij} \geq 0$, for all $ij \in I$.
\end{assumption}

\begin{assumption}[Assumption 6 in  \cite{de2014branch}]
\label{asmp: 4}
$a_{i1} \geq \ldots \geq a_{in_i},$ for all $i \in M - M_0.$
\end{assumption}

All these assumptions can be made without loss of generality (w.l.o.g.): If the first two assumptions do not hold, then problem~\eqref{CKP} is trivial. 
For Assumption~\ref{asmp: 3}, if $b_{ij} \leq 0$ for some index $ij$, then $x_{ij}$ can be fixed to zero. Assumption~\ref{asmp: 4} is made because we are dealing with a single knapsack inequality, and we make this assumption to simplify the presentation of this paper.

Next, we introduce a few basic properties of $PS$, and review the concepts and results introduced in de Farias et al.~\cite{de2002facets, de2014branch}. Throughout this section, known results are not proved, but referenced. The results that are proved are, to the best of our knowledge, new. 

\begin{lemma}[Proposition~2 in \cite{de2002facets}]
$PS$ is full-dimensional.
\end{lemma}

\begin{lemma}
\label{prop: CK_vertex_simple}
If $\tilde x$ is a non-integral vertex of $PS$, then it has exactly one fractional component, and we have $ \sum_{i \in M} \sum_{j \in N_i} a_{ij} \tilde x_{ij} = b$.
\end{lemma}
\begin{proof}
If $\tilde x$ has at least two fractional components, say, $\tilde x_{i' j'}, \tilde x_{i'' j''} \in (0,1)$, for some $i' \neq i'' \in M$, $j' \in N_{i'}$, $j'' \in N_{i''}$. 
Then the 2-dimensional point $(\tilde x_{i' j'}, \tilde x_{i'' j''})$ is contained in the 2-dimensional polyhedron 
$$
\bra{ (x_{i'j'}, x_{i''j''}) \in [0,1]^2 \mid a_{i'j'} x_{i'j'} + a_{i''j''} x_{i''j''} \leq b - \sum_{i \in M} \sum_{j \in N_i} a_{ij} \tilde x_{ij} + a_{i'j'} \tilde x_{i'j'} + a_{i''j''} \tilde x_{i''j''}}.
$$
It is simple to check that the extreme points of this 2-dimensional polyhedron all have at most one fractional component, therefore $(\tilde x_{i' j'}, \tilde x_{i'' j''})$ can be written as the convex combination of some other 2-dimensional points in the above polyhedron. 
Note that replacing the $x_{i'j'}$ and $x_{i''j''}$ components of $\tilde x$ with the components of these 2-dimensional points will lead to vectors in $S$, and $\tilde x$ can be written as the convex combination of these resulting points.
This gives us a contradiction, since $\tilde x$ is an extreme point of $PS$  by assumption.

We now assume, without loss of generality, that $\tilde x_{11}$ is the only fractional component of vertex $\tilde x$. 
If $ \sum_{i \in M} \sum_{j \in N_i} a_{ij} \tilde x_{ij} < b$, then replacing the $x_{11}$ component of $\tilde x$ by $\tilde x_{11} + \epsilon$ or $\tilde x_{11} - \epsilon$, for some small enough $\epsilon$, will still satisfy the inequality $ \sum_{i \in M} \sum_{j \in N_i} a_{ij} x_{ij} \leq b$ and the complementarity constraints. But this also implies that $\tilde x$ is not a vertex of $PS$, which gives a contradiction.
\end{proof}

 \emph{Lifting} is an important concept in integer programming that is  used to strengthen valid inequalities \cite{balas1975facets, balas1978facets, crowder1983solving, gu1998lifted, gu2000sequence}. 
In the case of 0-1 knapsack problems, the \emph{cover inequalities} are one of the most well-known family of cuts.
Applying the lifting process to cover inequalities produces the widely used \emph{lifted cover inequalities}, discovered independently by Balas \cite{balas1975facets} and Wolsey \cite{wolsey1975faces}. (See also the survey \cite{atamturk2005cover}).
%However, in LPCCs, these inequalities are no longer valid, and for that reason 
In \cite{de2002facets}, the authors extended the concept of a cover and cover inequalities to CKPs as follows:
\begin{definition}[Definition~1 in \cite{de2002facets}]
Let $C = \{i_1j_1, \ldots, i_k j_k\} \subset I, $ where $i_1, \ldots, i_k$ are all distinct. The set $C$ is called a \emph{cover} if $\sum_{ij \in C} a_{ij} > b$. Given a cover $C$, the inequality 
\begin{equation}
\sum_{ij \in C} a_{ij} x_{ij} \leq b
\end{equation}
is called a cover inequality.
\end{definition}
% \begin{definition}[Definition~2 in \cite{de2002facets}]
% \label{def: CI_comp}
% Let $\tilde x \in S$. Let $C = \{i_1 j_1, \ldots, i_k j_k\} \subset I$, where $i_1, \ldots, i_k$ are all distinct, and 
% $$
% \tilde x_{ij} = 0 \quad \forall ij \in I - C \text{ with } i \in M_C. 
% $$
% Let $F_0 = \{ij \in I - C \mid \tilde x_{ij} = 0\}$, $F_1 = \{ij \in I - C \mid \tilde x_{ij} = 1\}$, and $F_2 = \{ij \in I - C \mid \tilde x_{ij} \in (0,1)\}$. We say that $C$ is a \emph{cover} for $PS \cap \{x \in \R^d \mid x_{ij} = \tilde x_{ij} \ \forall ij \in I - C\}$ if
% $$
% \sum_{ij \in C}a_{ij} > b - \sum_{ij \in F_1 \cup F_2} a_{ij} \tilde x_{ij}.
% $$
% The inequality 
% \begin{equation}
% \label{eq: CI_comp}
% \sum_{ij \in C} a_{ij} x_{ij} \leq b - \sum_{ij \in F_1 \cup F_2} a_{ij} \tilde x_{ij}
% \end{equation}
% is called a \emph{cover inequality} for $PS \cap \{x \in \R^d \mid x_{ij} = \tilde x_{ij} \ \forall ij \in F_0 \cup F_1 \cup F_2\}$.
% \end{definition}
% As abbreviated in \cite{de2002facets}, when $\tilde x_{ij} = 0$,  $\forall ij \in I - C$, then $C$ is simply referred to as a \emph{cover}, and the inequality $\sum_{ij \in C} a_{ij} x_{ij} \leq b$ is referred to as a \emph{cover inequality}. 

Next, we present two families of cutting-planes for $PS$ proposed in \cite{de2014branch}, which are obtained by sequentially lifting the cover inequalities with respect to covers which satisfy different conditions.
The obtained inequalities generalize the two families of inequalities given in \cite{de2002facets}. 

%\begin{theorem}[Theorem~1 in \cite{de2002facets}]
%\label{theo: theo1_de2002}
%Let $C$ be a cover, and suppose that $j = 1 \ \forall ij \in C$. Assume that $C$ satisfies
%\begin{equation}
%\sum_{ij \in C - i'j'} a_{ij} + a_{i'j''} < b
%\end{equation}
%for some $i'j' \in C$ and $j'' \in N_{i'} - j'$. Then
%\begin{equation}
%\label{facet_de_1}
%\sum_{i \in M_C} a_{i1} x_{i1} + \sum_{i \in M_C} \sum_{j \in N_i - 1} \max\{a_{ij}, b - \sum_{k \in M_C - i} a_{k1}\} x_{ij} \leq b
%\end{equation}
%is valid and facet-defining for $PS$.
%\end{theorem}

\begin{theorem}[Theorem~3 in \cite{de2014branch}]
\label{theo: theo1_de2002}
Let $C$ be a cover. Denote the elements of $C$ as $ir_i$, $\forall i \in M_C$. 
Assume that $\sum_{i \in M_C-i'} a_{i r_i} + a_{i'j'} < b$ for some $i' \in M_C$ and $j' \in \{j \in N_{i'} \mid r_{i'} < j' \leq n_{i'}\}$. Then,
\begin{equation}
\label{facet_de_1}
\sum_{i \in M_C} \sum_{j=1}^{r_i - 1}a_{i r_i} x_{ij} + \sum_{i \in M_C} \sum_{j=r_i}^{n_i} \max \left\{a_{ij}, b - \sum_{k \in M_C - i} a_{k r_k} \right\} x_{ij} \leq b
\end{equation}
is valid for $PS$. Inequality~\eqref{facet_de_1} is facet-defining when $r_i = 1$, $\forall i \in M_C$. 
\end{theorem}

%\begin{theorem}[Theorem~2 in \cite{de2002facets}]
%\label{theo: theo2_de2002}
%Let $C$ be a cover, and suppose that $j = n_i \ \forall i \in M_C - i'$ for some $i'$ and $i'j' \in C$. Assume that $C$ satisfies $a_{i'j'} + \sum_{i \in M_C - i'} a_{i n_i} > b$ and $\sum_{i \in M_C} a_{i n_i } < b$. Then
%\begin{equation}
%\label{facet_de_2}
%\begin{split}
%\sum_{j \in N_{i'}} & \max\{a_{i'j}, b - \sum_{k \in M_C - i'} a_{k n_k}\} x_{i'j} + \sum_{i \in M_C - i'} a_{in_i} x_{in_i} \\
%& + \sum_{i \in M_C - i'} \sum_{j \in N_i - n_i} a_{in_i} \max \Big\{1, \frac{a_{ij}}{b - \sum_{k \in M_C - i} a_{k n_k}}\Big\} x_{ij} \leq b
%\end{split}
%\end{equation}
%is valid and facet-defining for $PS$.
%\end{theorem}

\begin{theorem}[Theorem~4 in \cite{de2014branch}]
\label{theo: theo2_de2002}
Let $C$ be a cover and $i'j' \in C$ with $j' < n_{i'}$. 
Denote the elements of $C$ other than $i'j'$ as $i t_i$, $\forall i \in M_C - i'$. 
Suppose that $\sum_{i \in M_C - i'}a_{i t_i} + a_{i' n_{i'}} < b$. Then,
\begin{equation}
\label{facet_de_2}
\begin{split}
& \sum_{j \in N_{i'}}  \max \left\{a_{i'j}, b - \sum_{k \in M_C - i'} a_{k t_k} \right\} x_{i'j}  \\
& + \sum_{i \in M_C - i'} \left(\sum_{j =1}^{t_i} a_{i t_i} \max \left\{1, \frac{a_{ij}}{b - \sum_{k \in M_C - {i, i'}} a_{k t_k} - a_{i' n_{i'}}}\right\} x_{ij} + \sum_{j=t_i+1}^{n_i} a_{ij} x_{ij} \right) \leq b
\end{split}
\end{equation}
is valid for $PS$. Inequality~\eqref{facet_de_2} is facet-defining when $t_i = n_i$, $\forall i \in M_C - i'$.
\end{theorem}

%
%These inequalities were derived by sequentially lifting the cover inequality $\sum_{ij \in C} a_{ij} x_{ij} \leq b$, which is the cover inequality with $\tilde x_{ij} = 0$ for all $ij \notin C$ in Definition~\ref{def: CI_comp}. However, for subsequent lifting, variables $\tilde x_{ij}$ could be fixed at any value between 0 and 1. The last main result of \cite{de2002facets} is that there is no loss of generality in fixing variables for subsequent lifting exclusively at 0.
%
%\begin{theorem}[Theorem~3 in \cite{de2002facets}]
%\label{theo: theo3_de2002}
%Let 
%\begin{equation}
%\label{eq: theo_3_facet}
%\sum_{ij \in I} \alpha_{ij} x_{ij} \leq \beta
%\end{equation}
%be a nontrivial facet-defining inequality for $PS$ obtained by sequentially lifting \eqref{eq: CI_comp}. Then, it is possible to obtain \eqref{eq: theo_3_facet} by sequentially lifting the cover inequality 
%$$
%\sum_{ij \in C \cup F_1 \cup F_2} a_{ij} x_{ij} \leq b,
%$$
%which is valid for $PS \cap \{x \in \R^d \mid x_{ij} = 0 \ \forall ij \in F_0\}$.
%
%\end{theorem}

\section{Separation complexity of lifted cover inequalities for CKP} 
\label{sec: complexity}

A fundamental component of cutting plane algorithms is separation---does there exist a cutting plane within a given family that is violated by a given infeasible solution?
For many classic classes of inequalities for integer programming, separation is known to be \NP-complete. 
In knapsack problems, \cite{klabjan1998complexity} showed that the separation for cover inequalities is \NP-complete, and \cite{MR1688131} showed the same result for lifted cover inequalities. 
Moreover, recently Del Pia et al.~\cite{MR4485516} further extended the  complexity result for $\{0,1\}$-knapsack problems to extended cover inequalities, $(1,k)$-configuration inequalities, and weight inequalities.
In \cite{de2014branch}, de Farias et al.~also conjectured that the separation problems for both \eqref{facet_de_1} and \eqref{facet_de_2} are \NP-complete. 
In this section, we show that this conjecture is true.

First, we formally define the separation problems for \eqref{facet_de_1} and \eqref{facet_de_2}. Here the LP relaxation of \eqref{CKP} is simply the optimization problem without the complementarity constraints~\eqref{constraint: complementarity}.

\bigskip
\noindent \textbf{Problem SP1}\\
\textbf{Input: }$(c,a,b) \in (\Z^d_+, \Z^d_+, \Z_+)$ and an optimal solution $x^*$ to the LP relaxation of~\eqref{CKP}.\\
\textbf{Question: }Is there a cover $C = \{i r_i \mid  i \in M_C\}$ of \eqref{CKP}, such that 
\begin{equation}
\label{q1}
\sum_{i \in M_C} \sum_{j=1}^{r_i - 1}a_{i r_i} x^*_{ij} + \sum_{i \in M_C} \sum_{j=r_i}^{n_i} \max \left\{a_{ij}, b - \sum_{k \in M_C - i} a_{k r_k} \right\} x^*_{ij} > b?
\end{equation}

\bigskip
\noindent \textbf{Problem SP2}\\
\textbf{Input: }$(c,a,b) \in (\Z^d_+, \Z^d_+, \Z_+)$ and an optimal solution $x^*$ to the LP relaxation of~\eqref{CKP}.\\
\textbf{Question: }Is there a cover $C = \{i t_i \mid i \in M_C - i'\} \cup \{i'j'\}$ of \eqref{CKP}, with $\sum_{i \in M_C - i'} a_{i t_i} + a_{i' n_{i'}} < b$, such that 
\begin{equation}
\label{q2}
\begin{split}
& \sum_{j \in N_{i'}}  \max \left\{a_{i'j}, b - \sum_{k \in M_C - i'} a_{k t_k} \right\} x_{i'j}  \\
& + \sum_{i \in M_C - i'} \left(\sum_{j =1}^{t_i} a_{i t_i} \max \left\{1, \frac{a_{ij}}{b - \sum_{k \in M_C - {i, i'}} a_{k t_k} - a_{i' n_{i'}}}\right\} x_{ij} + \sum_{j=t_i+1}^{n_i} a_{ij} x_{ij} \right) > b?
\end{split}
\end{equation}

Next, we show that Problem SP1 is \NP-complete.
%Now we show the separation complexity for inequality~\eqref{facet_de_1}.
The reduction is from the \emph{partition problem:}  Given $(\alpha_1, \ldots, \alpha_k; \beta) \in (\Z^k_+, \Z_+)$ with $\sum_{i=1}^k \alpha_i = 2\beta$, does there exist a subset $S \subseteq [k]$ such that $\sum_{i \in S}\alpha_i = \beta$?
The partition problem is one of the original 21 problems that Karp demonstrated to be \NP-complete \cite{karp:72}.

\begin{theorem}
\label{theo: complexity1}
Problem SP1 is \NP-complete.
\end{theorem}

\begin{proof}
Since verifying if a given point violates a given inequality can be done in polynomial time with respect to the input size of the point and the inequality, the separation problem SP1 is clearly in the class \NP. 
Therefore, we only have to show that SP1 is \NP-hard. 

Given an instance $(\alpha_1, \ldots, \alpha_k; \beta) \in (\Z^k_+, \Z_+)$ of the partition problem, we construct the following instance of \eqref{CKP}, where the data $(a,b,c, x^*)$ is defined as follows: 
\begin{align}
\label{eq: pp_reduction_1}
\begin{split}
& M := [k+1], \ n_1 = \ldots = n_k := 1, \ n_{k+1} := \beta+1,\\
& a_{i1} := \alpha_i \ \forall i \in [k], \ a_{k+1,1} := 3, \ a_{k+1,2} = \ldots = a_{k+1,\beta+1} := 1,\\
& b := \beta + 2, \ c := a,\\
& x^*_{11} = \ldots = x^*_{k1} := \frac{2\beta-3}{6 \beta}, \ x^*_{k+1,1} := 1, \ x^*_{k+1,2} = \ldots =  x^*_{k+1,\beta+1} := \frac{1}{3}.
\end{split}
\end{align} 
Since $\sum_{i=1}^n \alpha_i = 2\beta$, here $\sum_{i \in M}\sum_{j \in N_i} a_{ij} x^*_{ij} = b$ and $c = a$, so $x^*$ is an optimal solution to the LP relaxation of our constructed CKP instance. 
Hence, $(a,b,c,x^*)$ is a correct input to problem SP1, and the encoding size of $(a,b,c,x^*)$ is polynomial in the encoding size of $(\alpha_1, \ldots, \alpha_k; \beta)$. 

First, assume $S \subseteq [k]$ is a yes-certificate to the above partition problem, with $\sum_{i \in S}\alpha_i = \beta$. 
Then, we consider $C = \{i 1 \mid i \in S \cup \{k+1\}\}$. 
We have $\sum_{ij \in C}a_{ij} = \sum_{i \in S} \alpha_i + 3 = \beta + 3 = b+1$, so $C$ is a cover. Now we verify \eqref{q1}. 
By plugging $C$ and $x^*$ into the left-hand side of \eqref{q1}, we obtain: $\sum_{i \in S} \alpha_i \cdot \frac{2\beta-3}{6 \beta} + (3 \cdot 1 + \sum_{j=2}^{\beta+1} 2 \cdot \frac{1}{3}) = \beta + \frac{5}{2}$, which is larger than $b$. Hence, from a yes-certificate to the partition problem, we can obtain a yes-certificate to problem SP1 with the above constructed input data. 

Second, we assume $C$ is a cover to our constructed CKP, such that the corresponding inequality \eqref{facet_de_1} is not satisfied by $x^*$. 
The assumption in Theorem~\ref{theo: theo1_de2002} implies that there exists some $i' \in M_C$ and $j'$ such that $\sum_{i \in M_C-i'} a_{i r_i} + a_{i'j'} < b$, since otherwise inequality \eqref{facet_de_1} is dominated by the original knapsack constraint~\eqref{constraint: knapsack}.
In our constructed CKP instance, we know that $i'$ must be $k+1$ and $(k+1,1) \in C$. In other words, $\sum_{i \in M_C \cap [k]} \alpha_i + 1 < b$. Since $C$ is a cover, we also have $\sum_{i \in M_C \cap [k]} \alpha_i + 3 > b$. Here $\alpha_i \in \Z_+$ for all $i \in [k]$. Therefore, $\sum_{i \in M_C \cap [k]} = b-2 = \beta$, which means $M_C \cap [k]$ is a yes-certificate for the above partition problem. 

We have shown that there is a yes-certificate to SP1 with input $(a,b,c,x^*)$ in \eqref{eq: pp_reduction_1} if and only if there is a yes-certificate to the partition problem with input $(\alpha_1, \ldots, \alpha_k; \beta)$. 
Since the partition problem is \NP-hard, we obtain that SP1 is \NP-hard as well.
\end{proof}

Next, we consider Problem SP2.
The proof of the next theorem is almost identical to the proof of Theorem~\ref{theo: complexity1}.
In particular, it is obtained once again with a reduction from the partition problem with the data in \eqref{eq: pp_reduction_1}.
For completeness, we give a detailed proof.

\begin{theorem}
\label{theo: complexity2}
Problem SP2 is \NP-complete.
\end{theorem}

\begin{proof} 
As in the proof of Theorem~\ref{theo: complexity1}, it suffices to prove the \NP-hardness of SP2.
Given an instance $(\alpha_1, \ldots, \alpha_k; \beta)$ of the partition problem, we construct the instance of \eqref{CKP} with data $(a,b,c, x^*)$ according to \eqref{eq: pp_reduction_1}. 
As we have shown in the proof for Theorem~\ref{theo: complexity1}, $(a,b,c, x^*)$ is a correct input to SP2 with polynomial encoding size with respect to that of $(\alpha_1, \ldots, \alpha_k; \beta)$. 

First, assume $S \subseteq [k]$ is a yes-certificate to the above partition problem, with $\sum_{i \in S}\alpha_i = \beta$. 
Then, we consider $C = \{i 1 \mid i \in S \cup \{k+1\}\}$, and pick $(i',j') = (k+1,1)$. Here $\sum_{ij \in C}a_{ij} = \sum_{i \in S} \alpha_i + 3 = \beta + 3 = b+1$, so $C$ is a cover. 
Also, $\sum_{i \in M_C-\{k+1\}} a_{i1} + a_{k+1, \beta+1} = \sum_{i \in S} \alpha_i + 1 = \beta + 1 < b$.
Now we verify \eqref{q2}. By plugging $C$ and $x^*$ into the left-hand side of \eqref{q2}, we obtain: $(3 \cdot 1 + \sum_{j=2}^{\beta+1} 2 \cdot \frac{1}{3}) + \sum_{i \in S} \alpha_i \cdot \frac{2\beta-3}{6 \beta} = \beta + \frac{5}{2}$, which is larger than $b$. 
Hence, from a yes-certificate to the partition problem, we can obtain a yes-certificate to problem SP2 with input data $(a,b,c, x^*)$ constructed in \eqref{eq: pp_reduction_1}.

Second, we assume $C$ is a cover to our constructed CKP and $i'j' \in C$, such that 
$\sum_{i \in M_C - i'} a_{i t_i} + a_{i' n_{i'}} < b$,
and the corresponding inequality \eqref{facet_de_2} is not satisfied by $x^*$. In our constructed CKP instance, we know that $(i',j')$ must be $(k+1,1)$. In other words, $\sum_{i \in M_C \cap [k]} \alpha_i + 1 < b$. Since $C$ is a cover, we also have $\sum_{i \in M_C \cap [k]} \alpha_i + 3 > b$. Here $\alpha_i \in \Z_+$ for all $i \in [k]$. Therefore, $\sum_{i \in M_C \cap [k]} = b-2 = \beta$, which means $M_C \cap [k]$ is a yes-certificate for the above partition problem. 

We have shown that there is a yes-certificate to SP2 with input $(a,b,c,x^*)$ in \eqref{eq: pp_reduction_1} if and only if there is a yes-certificate to the partition problem with input $(\alpha_1, \ldots, \alpha_k; \beta)$. 
Since the partition problem is \NP-hard, we obtain that SP2 is \NP-hard as well.
\end{proof}

\section{New families of facet-defining inequalities}
\label{sec: cuts}

In this section, we propose three new families of inequalities valid for $PS$, that are fundamentally different from \eqref{facet_de_1} and \eqref{facet_de_2}. We also give sufficient conditions for the inequalities to be facet-defining.  The inequalities are derived using the concept of a \emph{pack}, a concept that is complementary to a cover, and we call the inequalities \emph{complementarity pack inequalities}.
We refer the reader to \cite{atamturk2005cover} for discussions on how packs are used to obtain strong valid inequalities in $0-1$ programming.

\begin{definition}
Let $P = \{i_1j_1, \ldots, i_k j_k\} \subset I, $ where $i_1, \ldots, i_k$ are all distinct. The set $P$ is called a \emph{pack} if $\sum_{ij \in P} a_{ij} < b$. 
A pack $P$ is further called a \emph{maximal switching pack}, if $j = n_i$ for any $ij \in P$, and $\sum_{ij \in P} a_{ij} - a_{i' n_{i'}} + a_{i', n_{i'} - 1} > b$, for any $i' \in M_P - M_0.$
\end{definition}

% All inequalities we introduce in this section will be arose from packs, hence we call them . 
% We remark that, these inequalities are completely irrelevant to the \emph{pack inequalities} (see, e.g., \cite{atamturk2005cover}) in MIPs.

\subsection{First complementarity pack inequalities}

Now we present the first family of complementarity pack inequalities.
\begin{theorem}
\label{theo: comp_ineq_1}
Let $P$ be a pack for $PS$, and $s := \sum_{ij \in P} a_{ij}$. 
Then
\begin{equation}
\label{ieq: 2}
\begin{split}
\sum_{i \in M_P} \sum_{j \in N_i} a_{ij} x_{ij} + \sum_{\substack{ij \in P,\\ i \in M_P - M_0}} (b-s) x_{ij} \leq b + (|M_P - M_0| - 1) (b-s)
\end{split}
\end{equation}
is a valid inequality for $PS$. 
Furthermore, when $P$ is a maximal switching pack, we have:
\begin{enumerate}
\item \eqref{ieq: 2} induces a face of $PS$ of dimension at least $d - |M_P - M_0|$;
\item When $M_P \cap M_0 \neq \emptyset$, \eqref{ieq: 2} is facet-defining.
\end{enumerate}
\end{theorem}

\begin{proof}
First, we show that any feasible point $\tilde x$ in $S$ satisfies \eqref{ieq: 2}. 
Let $\bar{M}: = M_P - M_0$ and $\bar P: = \{ij \in P \mid i \in \bar M\}$. 
We have $|\bar P| = |\bar M| = |M_P - M_0|$.

\smallskip \noindent
\textbf{Case 1}: $\bar P \subseteq \supp(\tilde x)$.
We know that,
%$P$ satisfies the property that, 
for any $i_1 j_1, i_2 j_2 \in P$, there is $i_1 \neq i_2$, and $\tilde x$ satisfies the complementarity constraint $\tilde x_{i j} \tilde x_{i, j'} = 0$ for any $j' \neq j \in N_i$, $i \in M$.
Therefore, in this case, $ \sum_{i \in M_P} \sum_{j \in N_i} a_{ij} \tilde x_{ij} \leq \sum_{ij \in P } a_{ij} = s$. Hence we have
\begin{align*}
\sum_{i \in M_P} \sum_{j \in N_i} a_{ij} \tilde x_{ij} +  (b-s) \sum_{ij \in \bar P} \tilde x_{ij} &
\leq \sum_{i \in M_P} \sum_{j \in N_i} a_{ij} \tilde x_{ij}  +  (b-s) |\bar P| \\
& \leq s +  (b-s) |\bar P| \\ 
& =  b + (|\bar P| - 1) (b-s) \\
& =  b + (|M_P - M_0| - 1) (b-s).
\end{align*}

\smallskip \noindent
\textbf{Case 2}: $\bar P \nsubseteq \supp(\tilde x)$. In this case, we have $|\{ij \mid ij \in \bar P,  ij \in \supp(\tilde x)\}| \leq  |\bar P | - 1 = |M_P - M_0| - 1$. Hence
\begin{align*}
 \sum_{i \in M_P} \sum_{j \in N_i} a_{ij} \tilde x_{ij} +  (b-s) \sum_{ij \in \bar P} \tilde x_{ij}
& \leq b  +  (b-s) |\{ij \mid ij \in \bar P, ij \in \supp(\tilde x)\}| \\
& \leq b + (|M_P - M_0| - 1) (b-s).
\end{align*}

From the discussion in the above two cases, we have shown that \eqref{ieq: 2} is satisfied by any feasible point $\tilde x \in S$, which means \eqref{ieq: 2} is valid for $PS$.

Next, we assume that $P$ is a maximal switching pack. 
A point $x$ with $x_{ij} = 1 \ \forall ij \in P$, and 0 elsewhere, satisfies \eqref{ieq: 2} at equality, and is in $S$, since by definition of pack we have $\sum_{ij \in P} a_{ij} < b$. 
Also, for any $i' \notin M_P$, $j' \in N_{i'}$, the point $x$ with $x_{i'j'} = \min\{1, \frac{b-s}{a_{i'j'}}\}$, $x_{ij} = 1 \ \forall ij \in P$, and $0$ elsewhere, is in $S$ and satisfies \eqref{ieq: 2} at equality. 
Hence so far we have found $\sum_{i \notin M_P} n_i + 1$ affinely independent feasible points of $S$ that satisfy \eqref{ieq: 2} at equality. 

Recall that $j = n_i$ for any $ij \in P$ when $P$ is a maximal switching pack. 
In the following, we discuss the dimension of the face induced by \eqref{ieq: 2}, and we consider separately two cases:

\smallskip \noindent
\textbf{Case 1:} $M_P \cap M_0 = \emptyset$.
For any $i' \in M_P$, $j' < n_{i'}$, consider the point $x$ with $x_{i n_i} = 1 \ \forall i \in M_P - i'$, $x_{i' j'} = \frac{a_{i' n_{i'}}  + b - s }{a_{i'j'}}$, and 0 elsewhere.
By definition of maximal switching pack and Assumption~\ref{asmp: 4}, we know $x_{i' j'} < 1$. 
It is easy to verify that $x$ is in $S$, and satisfies \eqref{ieq: 2} at equality. 
Therefore, we find another $\sum_{i \in M_P} (n_i - 1)$ points in $S$ that satisfy \eqref{ieq: 2} at equality.
Together with the previous $\sum_{i \notin M_P} n_i + 1$ points in $S$, we have found in total $\sum_{i \in M} n_i - |M_P| + 1 = d + 1 - |M_P|$ points of $S$ that satisfy \eqref{ieq: 2} at equality.
From our characterization of these points, we know that they are affinely independent. 
Therefore, the face given by \eqref{ieq: 2} has dimension at least $d - |M_P| = d - |M_P - M_0|$. 

\smallskip \noindent
\textbf{Case 2:} $M_P \cap M_0 \neq \emptyset$.
In this case we simply have $|M_P \cap M_0| \geq 1$.
For any $i' \in \bar M,$ we consider the following set
\begin{equation}
\label{eq: facet_set}
\begin{split}
\Big\{x \in [0,1]^d \mid \ & x_{in_i} = 1 \ \forall i \in \bar M - i', x_{ij} = 0 \ \forall i \in \bar M - i' \text{ and } j \in N_i - n_i,\\
& x_{i' n_{i'}} = 0, x_{ij} = 0 \ \forall i \notin M_P \text{ and } j \in N_i, \\
& \sum_{i \in M_P \cap M_0} a_{i1} x_{i1} + \sum_{j \in N_{i'} - n_{i'}} a_{i' j} x_{i' j} = b - \sum_{i \in \bar M - i'} a_{i n_i}, \\
& \text{ set } \{x_{i' 1}, \ldots, x_{i', n_{i'} - 1}\} \text{ is SOS1}
 \Big\}.
\end{split}
\end{equation}
By definition of maximal switching pack, we have $ \sum_{i \in M_P \cap M_0} a_{i1} + a_{i' j} > b - \sum_{i \in \bar M - i'} a_{i n_i}$ for any $j \in N_{i'} - n_{i'}$, so the set \eqref{eq: facet_set} has $|M_P \cap M_0| + n_{i'} - 1$ affinely independent points, which are all in $S$ and satisfy \eqref{ieq: 2} at equality. 
Hence in total we obtain $\sum_{i \in \bar M} (|M_P \cap M_0| + n_i - 1) = \sum_{i \in \bar M} n_i + |\bar M| (|M_P \cap M_0| - 1)$ affinely independent points in $S$ that satisfy \eqref{ieq: 2} at equality. 
Notice that these points all lie in the affine space 
\begin{equation}
\begin{split}
\Big\{x \in \R^d \mid & \ x_{ij} = 0 \ \forall i \notin M_P \text{ and } j \in N_i, \sum_{i \in \bar M} x_{i n_i} = |\bar M| - 1, \\
& \sum_{i \in M_P \cap M_0} a_{i1} x_{i1} + \sum_{j \in N_{i'} - n_{i'}} a_{i' j} x_{i' j} = b - \sum_{i \in \bar M - i'} a_{i n_i}\Big\},
\end{split}
\end{equation}
whose dimension is $d - \sum_{i \notin M_P} n_i - 2 = \sum_{i \in M_P} n_i - 2$. Since $|M_P \cap M_0| \geq 1$, we know that
the number of previous points $\sum_{i \in \bar M} n_i + |\bar M| (|M_P \cap M_0| - 1)$ is at least $\sum_{i \in \bar M} n_i + |M_P \cap M_0| - 1$, which equals $\sum_{i \in M_P} n_i - 1$. 
Hence, we obtain another $\sum_{i \in M_P} n_i - 1$ affinely independent points in $S$ that satisfy \eqref{ieq: 2} at equality. 
Together with the previous $\sum_{i \notin M_P} n_i + 1$ points in $S$, in total we have found $d$ points in $S$ which satisfy \eqref{ieq: 2} at equality, and this means that inequality \eqref{ieq: 2} is facet-defining when $M_P \cap M_0 \neq \emptyset$.
%
%Hence we complete the proof of this theorem from the above two cases.
\end{proof}

Next, we provide two examples where inequality \eqref{ieq: 2} is either fact-defining, or it defines a face of high dimension.

\begin{example}
\label{exam: 1}
Let $M = [5]$, $n_1 = n_2 = n_3 = 1, n_4 = n_5 = 2$, and let the knapsack inequality in \eqref{CKP} be
\begin{equation}
\label{eq: example-defining}
2x_{11} + 4 x_{21} + 8 x_{31} + (10 x_{41} + 6 x_{42}) + (8 x_{51} + 4 x_{52}) \leq 21.
\end{equation}
Take $P_1 = \{(1,1), (3,1), (4,2), (5,2)\}$. Since $2 + 8 + 6 + 4 = 20 < 21$, and $ 2+ 8 + 10 + 4 > 21, 2 + 8 + 6 + 8 > 21$, we know $P_1$ is a maximal switching pack. Furthermore, $M_{P_1} \cap M_0 = \{1, 3\} \neq \emptyset$, therefore Theorem~\ref{theo: comp_ineq_1} gives facet-defining inequality $2x_{11} + 8 x_{31} + (10 x_{41} + 7 x_{42}) + (8 x_{51} + 5 x_{52}) \leq 22$.

Next, consider the pack $P_2 = \{(3,1), (4,2), (5,2)\}$.
Also $P_2$ is a maximal switching pack, with $M_{P_2} \cap M_0 \neq \emptyset$. 
$P_2$ gives us another facet-defining inequality $8 x_{31} + (10 x_{41} + 9 x_{42}) + (8 x_{51} + 7 x_{52}) \leq 24$.
$\hfill\diamond$
\end{example}

\begin{example}
\label{exam: 2}
Let $M = [4]$, $n_1 = 1, n_2 = n_3 = n_4 = 2$, and let the knapsack inequality in \eqref{CKP} be
\begin{equation}
\label{eq: example-defining_2}
2x_{11} + (14 x_{21} + 10 x_{22}) + (13 x_{31} + 9 x_{32}) + (9 x_{41} + 6 x_{42}) \leq 22.
\end{equation}
Take $P_1 = \{(1,1), (2,2), (3,2)\}$, which is a maximal switching pack, with $M_{P_1} \cap M_0 \neq \emptyset$. 
Hence, Theorem~\ref{theo: comp_ineq_1} gives facet-defining inequality $2 x_{11} +  (14 x_{21} + 11 x_{22}) + (13 x_{31} + 10 x_{32}) \leq 23.$ 

Now pick $P_2 = \{(2,2), (3,2)\}$, which is also a maximal switching pack, but with $M_{P_2} \cap M_0 = \emptyset$. 
The corresponding inequality \eqref{ieq: 2} is $(14 x_{21} + 13 x_{22}) + (13 x_{31} + 12 x_{32}) \leq 25$.
This inequality is not necessarily facet-defining, and Theorem~\ref{theo: comp_ineq_1} states that it induces a face of $PS$ of dimension at least $d-|M_{P_2}-M_0| = 7-2 = 5$.
In fact, the dimension of the corresponding face of $PS$ is exactly 5.
$\hfill\diamond$
\end{example}

For the first family of complementarity pack inequalities, it is natural to ask the same question as in Sect.~\ref{sec: complexity}: What is the separation complexity of the inequalities in the form of \eqref{ieq: 2}? For that we have the following conjecture.
\begin{conjecture}
\label{con 1}
For an arbitrary infeasible solution to \eqref{CKP}, it is \NP-complete to determine if there exists a separating inequality of the form \eqref{ieq: 2}.
\end{conjecture}
    
\subsection{Second complementarity pack inequalities}

The next theorem provides us with the second class of complementarity pack inequalities for $PS$ and provides a sufficient condition for them to be facet-defining.

\begin{theorem}
\label{theo: comp_ineq_2}
Let $P$ be a pack for $PS$ with $|M_P - M_0| \geq 2$, and $s := \sum_{ij \in P} a_{ij}$. 
For any $i^*j^* \in P$ with $i^* \notin M_0$ and $j^* = n_{i^*}$, the inequality
\begin{equation}
\label{ieq: 3}
\begin{split}
& \sum_{i \in M_P - i^*} \sum_{j \in N_i} a_{ij} x_{ij} + \sum_{\substack{ij \in P,\\ i \in M_P - M_0 - i^*}} (b-s) x_{ij} + a_{i^* j^*} x_{i^* j^*} \\ 
+ & \sum_{j \in N_{i^*} - j^*} a_{i^* j^*} \max\left\{1, \frac{a_{i^* j} }{ a_{i^* j^*} + b - s} \right\} x_{i^* j} \leq b + (|M_P - M_0| - 2) (b-s)
\end{split}
\end{equation}
is valid for $PS$. Furthermore, when $P$ is a maximal switching pack, \eqref{ieq: 3} is facet-defining.
\end{theorem}

\begin{proof}
First, let $\tilde x$ in $S$.
We show that inequality \eqref{ieq: 3} is valid for $\tilde x$. 
Here we denote by $\bar M: = \{i \in M_P \mid i \notin M_0, i \neq i^*\}$, $\bar P: = \{ij \in P \mid i \in \bar M\}$.
We have $|\bar P| = |\bar M| = |M_P - M_0| - 1$.

\smallskip \noindent
\textbf{Case 1:} $\bar P \subseteq \supp(\tilde x)$.
In this case, we have $ \sum_{i \in M_P, i \neq i^*} \sum_{j \in N_i} a_{ij} \tilde x_{ij} + a_{i^* j^*}  \leq \sum_{ij \in P} a_{ij} = s$.

\smallskip \noindent
\textbf{Case 1a:} $\tilde x_{i^* j^*} > 0$. Because $\tilde x$ satisfies the complementarity consrtaint, we know that $\tilde x_{i^* j} = 0$ for all $j \in N_{i^*} - j^*$.
So the left-hand side of \eqref{ieq: 3} evaluated in $\tilde x$ is
\begin{align*}
\sum_{\substack{i \in M_P, \\ i \neq i^*}} \sum_{j \in N_i} a_{ij} \tilde x_{ij} + \sum_{\substack{ij \in P,\\ i \notin M_0, i \neq i^*}} (b-s) \tilde x_{ij} + a_{i^* j^*} \tilde x_{i^* j^*} 
& \leq s +  (b-s) |\bar P| \\ 
& = s + (|M_P - M_0| - 1) (b-s) \\
& = b + (|M_P - M_0| - 2) (b-s).
\end{align*}

\smallskip \noindent
\textbf{Case 1b:} There exists $j' \in N_{i^*} - j^*$ such that $\tilde x_{i^* j'} \in (0,1]$. Since $\tilde x \in S$, we have $a_{i^* j'} \tilde x_{i^* j'} \leq b - \sum_{i \in M_P, i \neq i^*} \sum_{j \in N_i} a_{ij} \tilde x_{ij}$. 
Let $K: = \sum_{i \in M_P, i \neq i^*} \sum_{j \in N_i} a_{ij} \tilde x_{ij}$.
We then have $K \leq s - a_{i^* j^*}$. 
Thus, the left-hand side of \eqref{ieq: 3} evaluated in $\tilde x$ is
\begin{align*}
&  \sum_{\substack{i \in M_P, \\ i \neq i^*}} \sum_{j \in N_i} a_{ij} \tilde x_{ij} + \sum_{\substack{ij \in P,\\ i \notin M_0, i \neq i^*}} (b-s) \tilde x_{ij} + a_{i^* j^*} \max \left\{1, \frac{a_{i^* j'}}{ a_{i^* j^*} + b - s} \tilde x_{i^* j'}  \right\}\\
& \leq K + (|M_P - M_0| - 1)(b-s) + a_{i^* j^*} \max \left\{1, \frac{b - K}{a_{i^* j^*} + b - s} \right\} \\
& = \max \left\{K + a_{i^* j^*}, K + a_{i^* j^*} \frac{b - K}{a_{i^* j^*} + b - s} \right\} + (|M_P - M_0| - 1)(b-s) \\
& \leq \max \left\{s, \frac{K(b-s) + a_{i^* j^*} b}{a_{i^* j^*} + b - s} \right\} +  (|M_P - M_0| - 1)(b-s) \\
& \leq \max \left\{s, \frac{(s - a_{i^* j^*})(b-s) + a_{i^* j^*} b}{a_{i^* j^*} + b - s} \right\} +  (|M_P - M_0| - 1)(b-s) \\
& = \max\{s, s\} +   (|M_P - M_0| - 1)(b-s) \\
& = b +   (|M_P - M_0| - 2)(b-s).
\end{align*}

\smallskip \noindent
\textbf{Case 1c:} $\tilde x_{i^* j} = 0$ for all $j \in N_{i^*}$.
The left-hand side of \eqref{ieq: 3} evaluated in $\tilde x$ is
\begin{align*}
\sum_{\substack{i \in M_P, \\ i \neq i^*}} \sum_{j \in N_i} a_{ij} \tilde x_{ij} + \sum_{\substack{ij \in P,\\ i \notin M_0, i \neq i^*}} (b-s) \tilde x_{ij}
& \leq s +   (|M_P - M_0| - 1)(b-s) \\
& = b +   (|M_P - M_0| - 2)(b-s).
\end{align*}

\smallskip \noindent
\textbf{Case 2:} $\bar P \nsubseteq \supp(\tilde x)$. 
In this case, we have $|\{i j \mid ij \in \bar P, ij \in \supp(\tilde x)\}| \leq |\bar P| - 1 = |M_P - M_0| - 2$. 
Note that $j^* = n_{i^*}$, and $a_{i^* j^*} \max\{1, \frac{a_{i^* j} }{ a_{i^* j^*} + b - s}\} \leq a_{i^* j}$ for all $j \in N_{i^*} - j^*$, so the left-hand side of \eqref{ieq: 3} evaluated in $\tilde x$ is at most 
\begin{align*}
b +  (b-s)|\{i j \mid ij \in \bar P, ij \in \supp(\tilde x)\}| 
\leq b +   (|M_P - M_0| - 2)(b-s).
\end{align*}
%\begin{enumerate}
%\item If $\tilde x_{i^* j^*} > 0$: Then 
%\begin{align*}
%& \text{ The left-hand side of }\eqref{ieq: 3} \text{ at } \tilde x\\ 
%= & \sum_{\substack{i \in M_P, \\ i \neq i^*}} \sum_{j \in N_i} a_{ij} \tilde x_{ij} + \sum_{\substack{ij \in P,\\ i \notin M_0, i \neq i^*}} (b-s) \tilde x_{ij} + a_{i^* j^*} \tilde x_{i^* j^*}\\
%\leq & b +  (b-s)|\{i j \mid ij \in \bar P, ij \in \supp(\tilde x)\}| \\
%\leq & b +   (|M_P - M_0| - 2)(b-s).
%\end{align*}
%\item If there exists $j' \in N_{i^*} - j^*$ such that $\tilde x_{i j'} > 0$: 
%\begin{align*}
%& \text{ The left-hand side of }\eqref{ieq: 3} \text{ at } \tilde x\\ 
%= & \sum_{\substack{i \in M_P, \\ i \neq i^*}} \sum_{j \in N_i} a_{ij} \tilde x_{ij} +  \max\{a_{i^* n_{i^*}} \tilde x_{i^* j'}, a_{i^* j'} \tilde x_{i^* j'} \frac{a_{i^* j^*}}{a_{i^* j^*} + b - s}\}  + \sum_{\substack{ij \in P,\\ i \notin M_0, i \neq i^*}} (b-s) \tilde x_{ij} \\
%\leq & \sum_{\substack{i \in M_P, \\ i \neq i^*}} \sum_{j \in N_i} a_{ij} \tilde x_{ij} +  \max\{a_{i^* n_{i^*}} \tilde x_{i^* j'}, a_{i^* j'} \tilde x_{i^* j'} \}  + \sum_{\substack{ij \in P,\\ i \notin M_0, i \neq i^*}} (b-s) \tilde x_{ij} \\
%= & \sum_{\substack{i \in M_P, \\ i \neq i^*}} \sum_{j \in N_i} a_{ij} \tilde x_{ij} +  a_{i^* j'} \tilde x_{i^* j'}  + \sum_{\substack{ij \in P,\\ i \notin M_0, i \neq i^*}} (b-s) \tilde x_{ij} \\
%\leq & b +  (b-s)|\{i j \mid ij \in \bar P, ij \in \supp(\tilde x)\}| \\
%\leq & b +   (|M_P - M_0| - 2)(b-s).
%\end{align*}
%\item If $\tilde x_{i^* j} = 0$ for all $j \in N_{i^*}$: Trivial, same as before.
%\end{enumerate}

So far we have concluded that \eqref{ieq: 3} is valid for $PS$. 
Next we show that, if $P$ is a maximal switching pack, then \eqref{ieq: 3} is facet-defining.

Consider the point $x$ with $x_{ij }= 1 \ \forall ij \in P,$ and 0 elsewhere. 
It is easy to check that $x$ is in $S$ and satisfies \eqref{ieq: 3} at equality. 
Also, for any $i' \in M_P$, $j' \in N_{i'}$, consider the point $x$ with $x_{ij }= 1 \ \forall ij \in P$, $x_{i'j'} = \min\{1, \frac{b - s}{a_{i'j'}}\},$ and 0 elsewhere. 
Also this point is in $S$, and satisfies \eqref{ieq: 3} at equality. 
So in total, we obtain $\sum_{i \notin M_P} n_i + 1$ affinely independent points in $S$, which all satisfy \eqref{ieq: 3} at equality.

Next, for any $j' \in N_{i^*} - n_{i^*}$, consider the point $x$ with $x_{ij} = 1 \ \forall ij \in P$, $x_{i^* j'} = \frac{a_{i^* n_{i^*}} + b - s}{a_{i^* j'}}$, and 0 elsewhere. 
Since $P$ is a maximal switching pack and $i^* \notin M_0$, we have $s - a_{i^* n_{i^*}} + a_{i^* j'} > b$, so $ x_{i^* j'} = \frac{a_{i^* n_{i^*}} + b - s}{a_{i^* j'}} < 1$, and $x \in PS$.
The left-hand side of \eqref{ieq: 3} evaluated in $\tilde x$ is 
\begin{align*}
& s - a_{i^* n_{i^*}} + (|M_P - M_0| - 1)(b-s) + a_{i^* n_{i^*}} \cdot \frac{a_{i^* j'}}{a_{i^* n_{i^*}} + b - s} \cdot \frac{a_{i^* n_{i^*}} + b - s}{a_{i^* j'}} \\
& = b + (|M_P - M_0| - 2)(b-s).
\end{align*}
Therefore, we have found $n_{i^*} - 1$ points in $S$ that satisfy \eqref{ieq: 3} at equality.

Recall that when $P$ is maximal switching pack, then for any $ij \in P$, we have $j = n_i$.
Arbitrarily pick $i' \in \bar M$, and consider the following set:
\begin{equation}
\label{eq: facet_set_2}
\begin{split}
\Big\{x \in [0,1]^d \mid \ & x_{ij} = 0 \ \forall i \notin M_P \text{ and } j \in N_i, x_{i' n_{i'}} = 0 ,\\
& x_{i n_i} = 1 \text{ and } x_{ij} = 0\ \forall i \in \bar M - i' \text{ and } j < n_i,  x_{i^* j} = 0 \ \forall j < n_{i^*},\\
& \sum_{i \in M_P \cap M_0} a_{i1} x_{i1} + a_{i^* n_{i^*}} x_{i^* n_{i^*}} + \sum_{j < n_{i'}} a_{i' j} x_{i' j} = b - \sum_{i \in \bar M - i'} a_{i n_i},\\
& \text{ set } \{x_{i' 1}, \ldots, x_{i', n_{i'} - 1}\} \text{ is SOS1}
 \Big\}.
\end{split}
\end{equation}
It is simple to verify that any point in the set \eqref{eq: facet_set_2} satisfies the complementarity constraint, as well as the knapsack constraint~\eqref{constraint: knapsack} (in fact, it is satisfied at equality), and it satisfies \eqref{ieq: 3} at equality. 
Since $P$ is a maximal switching pack, we have $\sum_{i \in M_P} a_{i n_i} - a_{i' n_{i'}} + a_{i' j} > b$ for any $j < n_{i'}$, so $\sum_{i \in M_P \cap M_0} a_{i1} + a_{i^* n_{i^*}} +  a_{i' j}  > b - \sum_{i \in \bar M - i'} a_{i n_i}$ for any $j < n_{i'}$. 
Therefore, the set \eqref{eq: facet_set_2} contains $|M_P \cap M_0| + n_{i'}$ affinely independent points which all satisfy \eqref{ieq: 3} at equality.
Lastly, for any $i'' \in \bar M - i'$, we consider the set 
\begin{equation}
\label{eq: facet_set_3}
\begin{split}
\Big\{x \in [0,1]^d \mid \ & x_{ij} = 0 \ \forall i \notin M_P \text{ and } j \in N_i, x_{i'' n_{i''}} = 0, x_{i 1} = 1 \ \forall i \in M_P \cap M_0,\\
& x_{i n_i} = 1 \text{ and } x_{ij} = 0\ \forall i \in \bar M - i'' \text{ and } j < n_i,  x_{i^* j} = 0 \ \forall j < n_{i^*},\\
& a_{i^* n_{i^*}} x_{i^* n_{i^*}} + \sum_{j < n_{i''}} a_{i'' j} x_{i'' j} = b - s + a_{i^* n_{i^*}} + a_{i'' n_{i''}}, \\
& \text{ set } \{x_{i'' 1}, \ldots, x_{i'', n_{i''} - 1}\} \text{ is SOS1}
 \Big\}.
\end{split}
\end{equation}
The points in the set \eqref{eq: facet_set_3} are in $S$, and satisfy \eqref{ieq: 3} at equality.
Furthermore, we can find $n_{i''}$ affinely independent points in the set \eqref{eq: facet_set_3}.

In total, we have found $\sum_{i \notin M_P} n_i + 1 + n_{i^*} - 1 + |M_P \cap M_0 | + n_{i'} + \sum_{i'' \in \bar M - i'} n_{i''} = d$ points in $S$ that satisfy \eqref{ieq: 3} at equality.
Furthermore, according to their construction, all these points are affinely independent. 
This concludes the proof that \eqref{ieq: 3} is facet-defining when $P$ is a maximal switching pack.
\end{proof}

%\note{In all the proof, you do not really prove the fact that the points are affinely independent.}

Next, we present examples illustrating that one can obtain several different facet-defining inequalities from the same maximal switching pack with different indices $i^*$.

%Next, we present some examples which illustrate that, with the same knapsack constraint, by picking different maximal switching packs and different indices $i^*$, we obtain several different facet-defining inequalities.

\begin{example}
\label{exam: 3}
Let $M = [5]$, $n_1 = n_2  = 1, n_3 = n_4 = n_5 = 2$, and let the knapsack inequality in \eqref{CKP} be
\begin{equation}
\label{eq: example-defining_3}
x_{11} + 6 x_{21} + (14 x_{31} + 10 x_{32}) + (13 x_{41} + 9 x_{42}) + (12 x_{51} + 8 x_{52}) \leq 36.
\end{equation}
For the maximal switching pack $P_1 = \{(1,1), (2,1), (3,2), (4,2), (5,2)\}$, we have $s = \sum_{ij \in P_1}a_{ij} = 1 + 6 + 10 + 9 + 8 = 34$. 
When $i^* = 3$, \eqref{ieq: 3} gives the facet-defining inequality $x_{11} + 6 x_{21}  + (13 x_{41} + 9 x_{42}) + (12 x_{51} + 8 x_{52}) + (36 - 34)\cdot (x_{42} + x_{52}) + 10 x_{32} + 10 \cdot \max\{1, \frac{14}{10 + 36 - 34}\} x_{31} \leq 36 + (36 - 34) \cdot (3-2)$, which can be simplified to
$$
x_{11} + 6 x_{21} + \pare{\frac{35}{3} x_{31} + 10 x_{32}} + (13 x_{41} + 11 x_{42}) + (12 x_{51} + 10 x_{52}) \leq 38.
$$
For the same maximal switching pack $P_1$, when $i^* = 4$ and $5$, inequality \eqref{ieq: 3} gives two other facet-defining inequalities: 
\begin{align*}
& x_{11} + 6 x_{21} + (14 x_{31} + 12 x_{32}) + \pare{\frac{117}{11} x_{41} + 9 x_{42}} + (12 x_{51} + 10 x_{52}) \leq 38, \\
& x_{11} + 6 x_{21} + (14 x_{31} + 12 x_{32}) + (13 x_{41} + 11 x_{42}) + \pare{\frac{48}{5} x_{51} + 8 x_{52}} \leq 38.
\end{align*}
Consider now the maximal switching pack $P_2 =  \{(2,1), (3,2), (4,2), (5,2)\}$.
Setting $i^* = 3$, $4$, and $5$ gives us the following three facet-defining inequalities, respectively:
\begin{align*}
& 6 x_{21} + \pare{\frac{140}{13} x_{31} + 10 x_{32}} + (13 x_{41} + 12 x_{42}) + (12 x_{51} + 11 x_{52}) \leq 39,\\
& 6 x_{21} + (14 x_{31} + 13 x_{32}) + \pare{\frac{39}{4} x_{41} + 9 x_{42}} + (12 x_{51} + 11 x_{52}) \leq 39, \\
& 6 x_{21} + (14 x_{31} + 13 x_{32}) + (13 x_{41} + 12 x_{42}) + \pare{\frac{96}{11} x_{51} + 8 x_{52}} \leq 39.
\end{align*}
$\hfill\diamond$
\end{example}

Similarly to Conjecture~\ref{con 1}, which we gave for the first complementarity pack inequalities, we pose the following conjecture regarding the separation complexity of inequalities of the form \eqref{ieq: 3}. 

\begin{conjecture}
For an arbitrary infeasible solution to \eqref{CKP}, it is \NP-complete to determine if there exists a separating inequality of the form \eqref{ieq: 3}.
\end{conjecture}

\subsection{Third complementarity pack inequalities}

Our final family of facet-defining complementarity pack inequalities is defined in Theorem~\ref{theo: comp_ineq_3}.
% Same as pack inequality~\eqref{ieq: 3}, here we obtain another similar class of facet-defining inequalities, which also depends on the specific choice of the index $i^*$.

\begin{theorem}
\label{theo: comp_ineq_3}
Let $P$ be a pack for $PS$ with $|M_P - M_0| \geq 2$, and $s:=\sum_{ij \in P} a_{ij}$.
For any $i^*j^* \in P$ with $i^* \notin M_0, j^* = n_{i^*}$ and any $i' \in M_P \cap M_0$, the inequality
\begin{equation}
\label{ieq: 4}
\begin{split}
 \frac{a_{i^* j^*} \cdot a_{i' 1}}{a_{i^* j^*} + b - s} x_{i' 1} + & \sum_{i \in M_P - \{i', i^*\}} \sum_{j \in N_i} a_{ij} x_{ij} + \sum_{\substack{ij \in P,\\ i \in M_P - M_0 - i^*}} \left(b-s + \frac{(b-s) a_{i'1}}{a_{i^* j^*} + b - s} \right) x_{ij} \\
+ & a_{i^* j^*} x_{i^* j^*}  +   \sum_{j \in N_{i^*} - j^*} a_{i^* j^*} \max \left\{1, \frac{a_{i^* j} }{ a_{i^* j^*} + b - s} \right\} x_{i^* j} \\
\leq & b + (|M_P - M_0| - 2) (b-s) \pare{1 +  \frac{ a_{i'1}}{a_{i^* j^*} + b - s}}
\end{split}
\end{equation}
is valid for $PS$. Furthermore, when $P - i' 1$ is a maximal switching pack, \eqref{ieq: 4} is facet-defining.
\end{theorem}

We remark that inequality \eqref{ieq: 4} is simply obtained by tilting the previous inequality \eqref{ieq: 3}: For a fixed index $i' \in M_0$, we: 
(i) subtract from the original coefficient $a_{i'1}$ in \eqref{ieq: 3} the quantity $(b-s)\frac{a_{i'1}}{a_{i^* j^* + b - s}}$, 
(ii) add the same amount $(b-s)\frac{a_{i'1}}{a_{i^* j^* + b - s}}$ to the coefficients of $x_{ij}$, for $ij \in P$ and $i \notin i^*$, and
(iii) multiply the right-hand side of the inequality by $(1 +  \frac{ a_{i'1}}{a_{i^* j^*} + b - s})$.

\begin{proof}[Proof of Theorem~\ref{theo: comp_ineq_3}]
Let $\tilde x$ be a vector in $S$.
We show that inequality \eqref{ieq: 4} is satisfied by $\tilde x$. 
For ease of exposition, we define $\bar M: = M_P - M_0 - i^*$ and $\bar P: = \{ij \in P \mid i \in \bar M\}$.
% then inequality \eqref{ieq: 4} is just
%\begin{equation}
%\label{eq: ieq_5_temp}
%\begin{split}
%& \sum_{i \in M_P} \sum_{j \in N_i} a_{ij} x_{ij} + \sum_{ij \in \bar P} \big(b-s + \frac{(b-s) a_{i'1}}{a_{i^* j^*} + b - s} \big) x_{ij} \\
%-  & \frac{(b-s) a_{i'1}}{a_{i^* j^*} + b - s} x_{i' 1} - \sum_{j \in N_{i^*} - j^*} \min\{a_{i^* j} - a_{i^* j^*}, \frac{(b-s)a_{i^* j}}{a_{i^* j^*} + b - s}\} x_{i^* j} \\
% \leq & b + (|M_P - M_0| - 2) (b-s) (1 +  \frac{ a_{i'1}}{a_{i^* j^*} + b - s}).
%\end{split}
%\end{equation}

\smallskip \noindent
\textbf{Case 1:} $\bar P \subseteq \supp(\tilde x)$.
In this case, since $\tilde x \in S$, we have $ \sum_{i \in M_P-i'} \sum_{j \in N_i} a_{ij} \tilde x_{ij} + a_{i' 1}  \leq \sum_{ij \in P} a_{ij} = s$.

\smallskip \noindent
\textbf{Case 1a:} $\tilde x_{i^* j} = 0$ for any $j \in N_{i^*}-j^*$.
In this case, the left-hand side of inequality \eqref{ieq: 4} evaluated in $\tilde x$ is
\begin{equation}
\begin{split}
& \sum_{i \in M_P-i'} \sum_{j \in N_i} a_{ij} \tilde x_{ij} +  \frac{a_{i^* j^*} a_{i'1}}{a_{i^* j^*} + b - s} \tilde x_{i' 1} + \sum_{ij \in \bar P} \left(b-s + \frac{(b-s) a_{i'1}}{a_{i^* j^*} + b - s} \right) \tilde x_{ij}  \\
& \leq s - a_{i' 1} +  \frac{a_{i^* j^*} a_{i'1}}{a_{i^* j^*} + b - s} + \left(b-s + \frac{(b-s) a_{i'1}}{a_{i^* j^*} + b - s} \right) (|M_P - M_0|  - 1) \\
& = b + (|M_P - M_0| - 2) (b-s) \pare{1 +  \frac{ a_{i'1}}{a_{i^* j^*} + b - s}}.
\end{split}
\end{equation}
So \eqref{ieq: 4} is satisfied by point $\tilde x$. 

\smallskip \noindent
\textbf{Case 1b:} There exists $j' \in N_{i^*} - j^*$ such that $\tilde x_{i^* j'} \in (0,1]$. 
Let $K := \sum_{i \in M_P-\{i^*,i'\}} \sum_{j \in N_i} a_{ij} \tilde x_{ij}$. 
Since $\bar P \subseteq \supp(\tilde x)$, we obtain $K \leq s - a_{i'1} - a_{i^* j^*}$, as well as $a_{i' 1} \tilde x_{i' 1} + a_{i^* j'} \tilde x_{i^* j'} + K \leq b$. Thus, the left-hand side of inequality \eqref{ieq: 4} evaluated in $\tilde x$ is
\begin{equation}
\begin{split}
& K + \frac{a_{i'1} a_{i^* j^*} \tilde x_{i'1}}{ a_{i^* j^*} + b - s} + \max \left\{a_{i^* j^*} \tilde x_{i^* j'}, \frac{a_{i^* j^*} a_{i^* j'}  \tilde x_{i^* j'}}{a_{i^* j^*} + b - s} \right\} \\
& \quad + \sum_{ij \in \bar P} \left(b-s + \frac{(b-s) a_{i'1}}{a_{i^* j^*} + b - s} \right) \tilde x_{ij} \\
& \leq \max \left\{K + \frac{a_{i'1} a_{i^* j^*} }{ a_{i^* j^*} + b - s} + a_{i^* j^*} , K + \frac{a_{i^* j^*} (a_{i' 1} \tilde x_{i' 1} + a_{i^* j'} \tilde x_{i^* j'})}{a_{i^* j^*} + b - s}  \right\} \\
& \quad + \left(b-s + \frac{(b-s) a_{i'1}}{a_{i^* j^*} + b - s} \right) (|M_P - M_0|  - 1)\\
& \leq \max \left\{K + \frac{a_{i'1} a_{i^* j^*} }{ a_{i^* j^*} + b - s} + a_{i^* j^*} , K + \frac{a_{i^* j^*} (b - K)}{a_{i^* j^*} + b - s}  \right\} \\
& \quad + (|M_P - M_0| - 1) (b - s) \pare{1 +  \frac{ a_{i'1}}{a_{i^* j^*} + b - s}}\\
& \leq \max \left\{ s - a_{i' 1} + \frac{a_{i'1} a_{i^* j^*} }{ a_{i^* j^*} + b - s}, \frac{ (b - s) K + a_{i^* j^*} b}{a_{i^* j^*} + b - s} \right\} \\
& \quad + (|M_P - M_0| - 1) (b - s) \pare{1 +  \frac{ a_{i'1}}{a_{i^* j^*} + b - s}}\\
& = s - a_{i' 1} + \frac{a_{i'1} a_{i^* j^*} }{ a_{i^* j^*} + b - s} + (|M_P - M_0| - 1) (b - s) (1 +  \frac{ a_{i'1}}{a_{i^* j^*} + b - s})\\
& = b + (|M_P - M_0| - 2) (b-s) \pare{1 +  \frac{ a_{i'1}}{a_{i^* j^*} + b - s}}.
\end{split}
\end{equation}

\smallskip \noindent
\textbf{Case 2:} $\bar P \nsubseteq \supp(\tilde x)$. 
In this case, we have $|\{i j \mid ij \in \bar P, ij \in \supp(\tilde x)\}| \leq |\bar P| - 1 = |M_P - M_0| - 2$. 
Note that $ \frac{a_{i^* j^*} \cdot a_{i' 1}}{a_{i^* j^*} + b - s} \leq a_{i' 1}$ and $a_{i^* j^*} \max\{1, \frac{a_{i^* j} }{ a_{i^* j^*} + b - s}\}  \leq a_{i^* j}$ for all $j \in N_{i^*} - j^*$, so the left-hand side of \eqref{ieq: 4} evaluated in $\tilde x$ is at most
\begin{align*}
& b +  (b-s)(1 +  \frac{ a_{i'1}}{a_{i^* j^*} + b - s})|\{i j \mid ij \in \bar P, ij \in \supp(\tilde x)\}| \\
& \leq b +   (|M_P - M_0| - 2)(b-s)\pare{1 +  \frac{ a_{i'1}}{a_{i^* j^*} + b - s}}.
\end{align*}

So far we have shown that inequality \eqref{ieq: 4} is valid for $PS$.
Next, we assume that $P - i'1$ is a maximal switching pack, and we show that \eqref{ieq: 4} is facet-defining. 
We note that, since $P - i' 1$ is a maximal switching pack and $P$ is a pack, we have that $P$ is also a maximal switching pack. 
Hence, for any $ij \in P$, $j = n_i.$

Consider the point $x$ with $x_{ij} = 1 \ \forall ij \in P$, and 0 elsewhere.
It is easy to check that $x \in S$ and it satisfies \eqref{ieq: 4} at equality. 
For any $i' \notin M_P$ and $j' \in N_{i'}$, consider the point $x$ with $x_{ij} = 1 \ \forall ij \in P$, $x_{i'j'} = \min\{1, \frac{b - s}{a_{i'j'}}\}$, and 0 elsewhere. 
Also this point is in $S$ and it satisfies \eqref{ieq: 4} at equality. 
Thus we have found $\sum_{i \notin M_P} n_i + 1$ points in $S$ that satisfy inequality \eqref{ieq: 4} at equality.

Next, we consider the following set
\begin{equation}
\label{eq: facet_set_4}
\begin{split}
\Big\{x \in [0,1]^d \mid \ & x_{i n_i} = 1, x_{ij} = 0 \ \forall i \in M_P - i' - i^*, j < n_{i},\\ 
& x_{i^* n_{i^*}} = 0, x_{ij} = 0 \ \forall i \notin M_P, j \in N_i, \\
& a_{i' 1} x_{i' 1} + \sum_{j < n_{i^*}} a_{i^* j} x_{i^* j} = b - s + a_{i' 1} + a_{i^* n_{i^*}}, \\
& \text{ set } \{x_{i^* 1}, \ldots, x_{i^*, n_{i^*} - 1}\} \text{ is SOS1}
 \Big\}.
\end{split}
\end{equation}
Since $P$ is a maximal switching pack, we have $s - a_{i^* n_{i^*}} + a_{i^* j} > b$ for any $j < n_{i^*}$. 
Hence, the set \eqref{eq: facet_set_4} has $n_{i^*}$ affinely independent points.
For any point $x$ in set \eqref{eq: facet_set_4}, the left-hand side of \eqref{ieq: 4} evaluated in $x$ is 
\begin{align*}
& \frac{a_{i^* n_{i^*}}  (b - s + a_{i' 1} + a_{i^* n_{i^*}})  }{a_{i^* n_{i^*}} + b - s}+ s - a_{i' 1} - a_{i^* n_{i^*}} \\
& \quad + (|M_P - M_0| - 1)(b-s)\pare{1 +  \frac{ a_{i'1}}{a_{i^* j^*} + b - s}} \\
& = s - a_{i' 1} + \frac{a_{i' 1} a_{i^* n_{i^*}}}{a_{i^* n_{i^*}} + b - s} + (|M_P - M_0| - 1)(b-s)\pare{1 +  \frac{ a_{i'1}}{a_{i^* j^*} + b - s}} \\ 
& = b +   (|M_P - M_0| - 2)(b-s)\pare{1 +  \frac{ a_{i'1}}{a_{i^* j^*} + b - s}}.
\end{align*}
It is simple to check that any point in \eqref{eq: facet_set_4} is in $S$. 
Hence, in total, we have found another $n_{i^*}$ affinely independent points in $S$ that satisfy \eqref{ieq: 4} at equality.

Now we arbitrarily pick an index $i'' \in \bar M$. 
Consider the following set:
\begin{equation}
\label{eq: facet_set_5}
\begin{split}
\Big\{x \in [0,1]^d \mid \ & x_{ij} = 0 \ \forall i \notin M_P \text{ and } j \in N_i, x_{i' 1} = 0, x_{i'' n_{i''}} = 0, \\
& x_{i n_i} = 1 \text{ and } x_{ij} = 0\ \forall i \in \bar M - i''  \text{ and } j < n_i,  x_{i^* j} = 0 \ \forall j < n_{i^*},\\
& \sum_{i \in M_P \cap M_0 - i' } a_{i1} x_{i1} + a_{i^* n_{i^*}} x_{i^* n_{i^*}} + \sum_{j < n_{i''}} a_{i'' j} x_{i'' j} = b - \sum_{i \in \bar M - i''} a_{i n_i},\\
& \text{ set } \{x_{i'' 1}, \ldots, x_{i'', n_{i''} - 1}\} \text{ is SOS1}
 \Big\}.
\end{split}
\end{equation}
It is easy to verify that any point in the set \eqref{eq: facet_set_5} satisfies the complementarity constraints in \eqref{constraint: complementarity}, as well as the knapsack constraint~\eqref{constraint: knapsack} (in fact, it is satisfied at equality), and it satisfies \eqref{ieq: 4} at equality. 
Since $P - i' 1$ is a maximal switching pack, we have $\sum_{i \in M_P - i'} a_{i n_i} - a_{i'' n_{i''}} + a_{i'' j} > b$ for any $j < n_{i''}$, therefore $\sum_{i \in M_P \cap M_0 - i' } a_{i1} + a_{i^* n_{i^*}} +  a_{i'' j}  > b - \sum_{i \in \bar M - i''} a_{i n_i}$ for any $j < n_{i'}$. 
Hence, the set \eqref{eq: facet_set_2} contains $|M_P \cap M_0| + n_{i''} - 1$ affinely independent points which all satisfy \eqref{ieq: 3} at equality.

Lastly, for any $\hat i \in \bar M - i''$, we consider the set 
\begin{equation}
\label{eq: facet_set_6}
\begin{split}
\Big\{x \in [0,1]^d \mid \ & x_{ij} = 0 \ \forall i \notin M_P \text{ and } j \in N_i, x_{i' 1} = 0,  x_{\hat i n_{\hat i}} = 0,\\
& x_{i n_i} = 1 \text{ and } x_{ij} = 0\ \forall i \in \bar M - \hat i \text{ and } j < n_i,  x_{i^* j} = 0 \ \forall j < n_{i^*},\\
& a_{i^* n_{i^*}} x_{i^* n_{i^*}} + \sum_{j < n_{\hat i}} a_{\hat i j} x_{\hat i j} = b - s + a_{i^* n_{i^*}} + a_{\hat i n_{\hat i}} + a_{i' 1}, \\
&  x_{i 1} = 1 \ \forall i \in M_P \cap M_0 - i', \text{ set } \{x_{\hat i 1}, \ldots, x_{\hat i, n_{\hat i} - 1}\} \text{ is SOS1}
 \Big\}.
\end{split}
\end{equation}
Again, the set \eqref{eq: facet_set_6} contains $n_{\hat i}$ affinely independent points in $S$, and they all satisfy \eqref{ieq: 4} at equality.

In total, we have found $\sum_{i \notin M_P} n_i + 1 + n_{i^*}  + |M_P \cap M_0 | + n_{i''} - 1 + \sum_{\hat i \in \bar M - i''} n_{\hat i} = d$ points in $S$ that satisfy \eqref{ieq: 4} at equality.
Furthermore, according to their construction, all these points are affinely independent. 
This concludes the proof that \eqref{ieq: 4} is facet-defining when $P - i' 1$ is a maximal switching pack.
\end{proof}

\begin{example}
Consider the same problem studied in Example~\ref{exam: 3}.
Consider $P = \{(1,1), (2,1), (3,2), (4,2), (5,2)\}$, and $i' = 1 \in M_P \cap M_0$. Here $s = \sum_{ij \in P} a_{ij} = 34$. Then both $P$ and $P - i' 1$ are maximal switching packs, satisfying the condition in Theorem~\ref{theo: comp_ineq_3}. For $i^* = 3$, \eqref{ieq: 4} gives the facet-defining inequality 
$$
\frac{5}{6} x_{11} + 6 x_{21} + \pare{\frac{35}{3} x_{31} + 10 x_{32}} + \pare{13 x_{41} + \frac{67}{6} x_{42}} + \pare{12 x_{51} + \frac{61}{6} x_{52}} \leq 38 + \frac{1}{6}.
$$
Similarly, picking $i^* = 4$ and 5, \eqref{ieq: 4} gives another two facet-defining inequalities:
\begin{align*}
& \frac{9}{11} x_{11} + 6 x_{21} + \pare{14 x_{31} + \frac{134}{11} x_{32}} + \pare{\frac{117}{11} x_{41} + 9 x_{42}} + \pare{12 x_{51} + \frac{112}{11} x_{52}} \leq 38 + \frac{2}{11}, \\
& \frac{4}{5} x_{11} + 6 x_{21} + \pare{14 x_{31} + \frac{61}{5} x_{32}} + \pare{13 x_{41} + \frac{56}{5} x_{42}} + \pare{\frac{48}{5} x_{51} + 8 x_{52}} \leq 38 + \frac{1}{5}.
\end{align*}
$\hfill\diamond$
\end{example}

We pose the following conjecture:
\begin{conjecture}
For an arbitrary infeasible solution to \eqref{CKP}, it is \NP-complete to determine if there exists a separating inequality in the form of \eqref{ieq: 4}.
\end{conjecture}

Here we briefly mention the main difference between our complementarity pack inequalities and the lifted cover inequalities \eqref{facet_de_1} and \eqref{facet_de_2}.
As discussed in \cite{de2014branch}, \eqref{facet_de_1} and \eqref{facet_de_2} are both obtained through sequential lifting of the original cover inequalities $\sum_{ij \in C}a_{ij}x_{ij} \leq b$. For a complete survey about lifting procedures, we refer the reader to Section 3 in \cite{hojny2019knapsack}. 
One fundamental property of all lifting procedures is that the difference between the original inequality and the lifted version of the inequality only occurs on the coefficients of variables that do not appear in the original inequality. 
In the cases of lifted cover inequalities \eqref{facet_de_1} and \eqref{facet_de_2}, one can easily verify that the coefficients of variables $x_{ij}$, for $ij \in C$, remain equal to $a_{ij}$. 
However, in the complementarity pack inequalities \eqref{ieq: 2}, \eqref{ieq: 3}, \eqref{ieq: 4}, the coefficients of variables $x_{ij}$, for $ij \in P$, all increase. This implies that our complementarity pack inequalities cannot be easily obtained through some lifting process of the inequality $\sum_{ij \in P}a_{ij} x_{ij} \leq b$.

%Unlike the inequalities in Theorem~\ref{theo: theo1_de2002} and \ref{theo: theo2_de2002}, here all our proposed facet-defining inequalities are originated from \emph{pack} instead of \emph{cover}. Other than this, 
%we want to briefly mention another major difference from those inequalities proposed in \cite{de2002facets}: \eqref{facet_de_1} and \eqref{facet_de_2} can be derived from the sequential lifting of the cover inequalities, while our proposed facets cannot be derived from such lifting process.
%
%\begin{theorem}
%The facet-defining inequalities in Theorem~\ref{theo: comp_ineq_1}, Theorem~\ref{theo: comp_ineq_2} and Theorem~\ref{theo: comp_ineq_3} cannot be obtained by sequentially lifting the cover inequality \eqref{eq: CI_comp}.
%\end{theorem}
%
%\note{re-write the following proof.}
%\begin{proof}
%Theorem~\ref{theo: theo3_de2002} states that, for any nontrivial facet-defining inequality $\sum_{ij \in I} \alpha_{ij} x_{ij} \leq \beta$ of $PS$, it can be obtained by sequentially lifting the cover inequality $\sum_{ij \in C \cup F_1 \cup F_2} a_{ij} x_{ij} \leq b$ that is valid for $PS \cap \{x \in \R^d \mid x_{ij} = 0 \ \forall ij \in F_0\}$. According to the lifting process, we know that $\beta = b$. But that is not the case for those inequalities in Theorem~\ref{theo: comp_ineq_1}, Theorem~\ref{theo: comp_ineq_2} and Theorem~\ref{theo: comp_ineq_3}.
%\end{proof}

\section{Exact separation algorithm}
\label{sec: heuristics}
In Sect.~\ref{sec: complexity}, we settled the separation complexity of two families of cutting-planes introduced in \cite{de2014branch}. Even though we cannot establish the separation complexity for the three complementarity pack inequalities that we proposed in Sect.~\ref{sec: cuts}, in this section we propose an efficient separation algorithm that runs in polynomial-time if $|M|$ is logarithmic in the input size of \eqref{CKP}.  
% As we remarked in Sect.~\ref{sec: complexity}, for $\{0,1\}$-knapsack problem, all the associated separation problems of cover inequalities, extended cover inequalities, $(1,k)$-configuration inequalities, weight inequalities \etc \ are all \NP-complete. However, for the benefits of computation in the practical branch-and-cut procedure, various efficient heuristics and exact separation algorithms are present in the literature for those families of cuts. 
% Both Gabrel and Minoux \cite{gabrel2002scheme} and Kaparis and Letchford \cite{kaparis2010separation} provide an exact separation algorithm for extended cover inequalities that runs in pseudo-polynomial time. Ferreira {\it et al.} \cite{ferreira1996solving} presented simple heuristics for the separation problem of $(1,k)$-configuration inequalities. For the separation problem for weight inequalities, Weismantel \cite{weismantel19970} proposed an exact algorithm that runs in pseudo-polynomial time. 
% Helmberg and Weismantel \cite{MR1607353} further presented a fast separation heuristic for weight inequalities that simply inserts items into the pack $P$ in non-increasing order of $x^*$ value, here $x^*$ is the fractional solution that is to be separated. Kaparis and Letchford \cite{kaparis2010separation} gave two exact algorithms and a heuristic for separating weight inequalities and show how to convert these methods into heuristics for separating lifted pack inequalities.
Throughout this section, assume that $x^*$ is an optimal solution obtained in the process of branch-and-cut, $x^*$ satisfies the knapsack constraint~\eqref{constraint: knapsack} while not being feasible in \eqref{CKP}: $x^*$ violates the complementarity constraint \eqref{constraint: complementarity}. 
Furthermore, w.l.o.g., we assume that $x^*$ satisfies all the \emph{clique inequalities}: $\sum_{j \in N_i} x^*_{ij} \leq 1$ for any $i \in M$, since otherwise we can simply add some clique inequality to separate $x^*$. In the following, we discuss how to separate such point $x^*$ using the cutting-planes that we devised in Sect.~\ref{sec: cuts}. 

We now consider the first family of complementarity pack inequalities \eqref{ieq: 2}. The goal here is to find a pack $P$ such that inequality \eqref{ieq: 2} is violated by $x^*$:
\begin{equation}
\label{eq: sep_1}
    \sum_{i \in M_P} \sum_{j \in N_i} a_{ij} x^*_{ij} + \sum_{\substack{ij \in P,\\ i \in M_P - M_0}} (b-s) x^*_{ij} > b + (|M_P - M_0| - 1) (b-s).
\end{equation}
Let $\bar{M}: = M_P - M_0$, $\bar P: = \{ij \in P \mid i \in \bar M\}$, and $P_0: = \{i0 \in P \mid i \in M_P \cap M_0\}$. Here $P = \bar P \cup P_0$.  
Let $x^*_{i,o_{i,1}} \geq x^*_{i,o_{i,2}} \geq \ldots \geq x^*_{i, o_{i,n_i}}$ for any $i \in M-M_0$, where $o_{i,1}, \ldots, o_{i,n_i}$ are different indices in $ [n_i]$. 
% Observed from the proof of the validity of \eqref{ieq: 2} to $PS$, here $\bar P$ has to be a subset of $\supp(x^*)$. If $\bar P \nsubseteq \supp(x^*)$, then $|\{ij \mid ij \in \bar P,  ij \in \supp(x^*)\}| \leq  |\bar P | - 1 = |M_P - M_0| - 1$, which naturally implies 
% \begin{align*}
%  \sum_{i \in M_P} \sum_{j \in N_i} a_{ij} x^*_{ij} +  (b-s) \sum_{ij \in \bar P} x^*_{ij}
% & \leq b  +  (b-s) |\{ij \mid ij \in \bar P, ij \in \supp(x^*)\}| \\
% & \leq b + (|M_P - M_0| - 1) (b-s),
% \end{align*}
% so \eqref{eq: sep_1} cannot be true. 
% Here the first inequality is assuming $x^*$ does not violate the knapsack constraint. 
Then we have the next simple result.
\begin{lemma}
\label{lem: sep_heu}
    If $P$ is a pack satisfying \eqref{eq: sep_1}, then $\sum_{ij \in \bar P} (1-x^*_{ij}) < 1,$ and $|\{ij \in \bar P \mid j \neq o_{i,1}\}| \leq 1$.   
\end{lemma}

\begin{proof}
Inequality $\sum_{ij \in \bar P} (1-x^*_{ij}) < 1$ holds naturally from
$$
\sum_{ij \in \bar P} (1-x^*_{ij}) < \frac{\sum_{i \in M_P}\sum_{j \in N_i} a_{ij} x^*_{ij} - s}{b-s} \leq 1.
$$
Here, the first inequality is from \eqref{eq: sep_1}, and the second inequality holds because $x^*$ satisfies the knapsack constraint~\eqref{constraint: knapsack}. 
Next, assume for a contradiction that there exist $i_1j_1, i_2j_2 \in \bar P$ with $j_1 \neq o_{i_1, 1}, j_2 \neq o_{i_2, 1}$. 
Since $x^*$ satisfies the clique inequalities, we have $2x^*_{i_1,j_1} \leq x^*_{i_1, o_{i_1, 1}} + x^*_{i_1,j_1} \leq 1$. Similarly, we also have $x^*_{i_2,j_2} \leq 1/2$. Then, $\sum_{ij \in \bar P} (1-x^*_{ij}) \geq (1-x^*_{i_1,j_1})+(1-x^*_{i_2,j_2}) \geq 1$, which contradicts the first inequality we have shown. 
\end{proof}

Now we focus on $P_0$. Let 
$$f(P): = \sum_{i \in M_P} \sum_{j \in N_i} a_{ij} x^*_{ij} - \sum_{ij \in \bar P} (b-s) (1-x^*_{ij})- s.$$
It is obvious to see that \eqref{eq: sep_1} holds if and only if $f(P) > 0$. Arbitrarily pick $i'0 \in P_0$. Then, basic algebra yields
\begin{equation}
\label{eq: basic_alg}
f(P) - f(P - i'0) = a_{i'0} \left( x^*_{i'0} + \sum_{ij \in \bar P} (1-x^*_{ij}) - 1 \right).
\end{equation}

We are now ready to propose, in Algorithm~\ref{alg:exact_sep}, the first exact separation procedure for \eqref{ieq: 2}.

\begin{algorithm}[H]
   \caption{Exact separation of \eqref{ieq: 2}}
   \label{alg:exact_sep}
       \hspace*{\algorithmicindent} \textbf{Input:} 
       An instance of \eqref{CKP} and an infeasible solution $x^*$. \\
       \hspace*{\algorithmicindent} \textbf{Output:} 
       A separating inequality in the form of \eqref{ieq: 2} if one exists.
    \begin{algorithmic}[1]
    \State Let $x^*_{i,o_{i,1}} \geq x^*_{i,o_{i,2}} \geq \ldots \geq x^*_{i, o_{i,n_i}}$ for any $i \in M-M_0$, where $o_{i,1}, \ldots, o_{i,n_i}$ are different indices in $ [n_i]$. 
    \For {$\bar M \subseteq \{i \in M-M_0 \mid x^*_{i,o_{i,1}} > 0\}$}
    \If {$\sum_{i \in \bar M} (1-x^*_{i,o_{i,1}})\geq 1$}
    \State \textbf{Continue}
    \Else
        \For{$i' \in \bar M$}
            \For{$j' = o_{i', 1}, o_{i', 2}, \ldots, o_{i', n_{i'}}$}
            \State Set $\bar P := i'j' \cup \{(i,o_{i,1}) \mid \forall i \in \bar M-i'\}$. 
            \If {$\sum_{ij \in \bar P} (1-x^*_{ij})\geq 1$}
    \State \textbf{Break}
    \ElsIf{$\bar P$ is not a pack}
    \State \textbf{Continue}
    \Else
    % \If{$\bar P \cup \{i0 \mid x^*_{i0} > 1- \sum_{ij \in \bar P} (1-x^*_{ij}), i \in M_0\}$ does not satisfy \eqref{eq: sep_1}}
    % \State \textbf{Continue}
    % \EndIf
    \For{$P_0 \subseteq \{i0 \mid x^*_{i0} > 1- \sum_{ij \in \bar P} (1-x^*_{ij}), i \in M_0\}$}
    \State Set $P:= \bar P \cup P_0$.
    \If{$P$ is a pack and \eqref{eq: sep_1} holds}
    \State \Return: pack $P$ gives a separating inequality \eqref{ieq: 2}.
    \EndIf
    \EndFor
    \EndIf
            \EndFor
        \EndFor
    \EndIf
    \EndFor
\end{algorithmic}
\end{algorithm}

The key intuition behind Algorithm~\ref{alg:exact_sep} is the following: if a certain index $i \in M-M_0$ is contained in $M_P$, then almost certainly $(i,o_{i,1}) \in P$. Here the exception cannot be made more than once according to Lemma~\ref{lem: sep_heu}. Therefore, the main complexity of the algorithm reduces to fixing the index of $\bar M$, which takes at most $O(2^{|M|-|M_0|})$ iterations.
This is much smaller than $O(\prod_{i \in M-M_0} (n_i + 1))$ iterations in a brute force method. 

The following theorem confirms the validity of Algorithm~\ref{alg:exact_sep}.
\begin{theorem}
    Algorithm~\ref{alg:exact_sep} is an exact separation algorithm for \eqref{ieq: 2}. 
\end{theorem}
\begin{proof}
Given the input to this algorithm, a separating inequality exists in the form of \eqref{ieq: 2} if and only if there exists a pack $P$ such that $f(P) > 0$. We show that Algorithm~\ref{alg:exact_sep} finds a pack $P$ with $f(P)>0$ if one exists. 
The procedure followed by the algorithm is: first, we fix $\bar M$, which is the set of indices in $M_P-M_0$. Here we can further assume $x^*_{i,o_{i,1}}$ to be positive for any $i \in \bar M$, since otherwise the condition of line 3 holds. From Lemma~\ref{lem: sep_heu} we know that, except for at most one $i' \in \bar M$, all other $i \in \bar M$ should have $(i,o_{i,1}) \in P$. So for line 6 and line 7, we iterate over all such indices $i' \in \bar M$ and $j' = o_{i', 1}, o_{i', 2}, \ldots, o_{i', n_{i'}}$, and this directly fixes the set $\bar P$. 
Now we assume that condition in line 9 holds. Since the inner loop iterates from $j' = o_{i', 1}$ to $j' = o_{i', n_{i'}}$ which is in decreasing order for $x^*_{i'j'}$, so if $\sum_{ij \in \bar P} (1-x^*_{ij}) < 1$ fails at one step it must also fail in the following steps within the inner loop. Hence it is safe for us to ``break" out of the inner loop in line 10. Once the set $\bar P$ is fixed, in line 14 we iterate over all possible sets for $P_0$. Here we only have to consider subsets in $\{i0 \mid x^*_{i0} > 1- \sum_{ij \in \bar P} (1-x^*_{ij}), i \in M_0\}$, because from \eqref{eq: basic_alg}, if $x^*_{i''0} \leq 1- \sum_{ij \in \bar P} (1-x^*_{ij})$ for some $i'' \in P$, then we can simply remove $i''$ from $P$ and this will not decrease the value of $f(P)$, while maintaining $P$ to be a pack. By iterating through all possible sets $\bar P$ and $P_0$, Algorithm~\ref{alg:exact_sep} will find a pack $P$ with $f(P)>0$ if one exists. 
\end{proof}

Denote by $n:=\sum_{i \in M} n_i$ the number of variables in \eqref{CKP}, and set $m:=|M|$. 
Compared with a brute force algorithm that runs in $O(\prod_{i \in M} (n_i + 1)) = O(n^m)$ iterations to find a valid pack $P$, Algorithm~\ref{alg:exact_sep} only runs in $O(2^{m} \cdot \sum_{i \in M-M_0} n_i) = O(n 2^m)$ iterations. So if $m$ is logarithmic in the input size, then Algorithm~\ref{alg:exact_sep} is a polynomial-time algorithm. 

Based on the same idea used to devise Algorithm~\ref{alg:exact_sep}, we can devise corresponding separation algorithms for the other two types of complementarity pack inequalities \eqref{ieq: 3} and \eqref{ieq: 4}. 
Similarly to Lemma~\ref{lem: sep_heu}, here we have the following result.
\begin{lemma}
\label{lem: sep_heu2}
    If $P$ is a pack and the corresponding inequality \eqref{ieq: 3} or \eqref{ieq: 4} separates $x^*$, then $\sum_{ij \in \bar P, i \neq i^*} (1-x^*_{ij}) < 1$ and $|\{ij \in \bar P \mid i \neq i^*, j \neq o_{i,1}\}| \leq 1$.   
\end{lemma}
\begin{proof}
    Assume that inequality \eqref{ieq: 3} separates $x^*$:
    \begin{equation*}
\begin{split}
& \sum_{i \in M_P - i^*} \sum_{j \in N_i} a_{ij} x^*_{ij} + \sum_{ij \in \bar P, i \neq i^*} (b-s) x^*_{ij} + a_{i^* j^*} x^*_{i^* j^*} \\ 
+ & \sum_{j \in N_{i^*} - j^*} a_{i^* j^*} \max\left\{1, \frac{a_{i^* j} }{ a_{i^* j^*} + b - s} \right\} x^*_{i^* j} > b + (|M_P - M_0| - 2) (b-s).
\end{split}
\end{equation*}
Note that $a_{i^* j^*} \max\left\{1, \frac{a_{i^* j} }{ a_{i^* j^*} + b - s} \right\} \leq a_{i^* j}$ for any $j \in N_{i^*} - j^*$, hence we have:
\begin{align*}
& \sum_{i \in M_P - i^*} \sum_{j \in N_i} a_{ij} x^*_{ij} + \sum_{ij \in \bar P, i \neq i^*} (b-s) x^*_{ij} + \sum_{j \in N_{i^*}} a_{i^* j} x^*_{i^* j} \\
> & b + (|M_P - M_0| - 2) (b-s)\\
= & s + (|\bar P| - 1) (b-s).
\end{align*}
Note that $\sum_{i \in M_P - i^*} \sum_{j \in N_i} a_{ij} x^*_{ij} + \sum_{j \in N_{i^*}} a_{i^* j} x^*_{i^* j} \leq b$. 
Therefore, $\sum_{ij \in \bar P, i \neq i^*} x^*_{ij} < 1$. Following the same argument as in the proof of Lemma~\ref{lem: sep_heu}, this inequality directly implies $|\{ij \in \bar P \mid i \neq i^*, j \neq o_{i,1}\}| \leq 1$. 

Now assume that \eqref{ieq: 4} separates $x^*$. Then:
\begin{equation*}
\begin{split}
 \frac{a_{i^* j^*} \cdot a_{i' 1}}{a_{i^* j^*} + b - s} x^*_{i' 1} + & \sum_{i \in M_P - \{i', i^*\}} \sum_{j \in N_i} a_{ij} x^*_{ij} + \sum_{ij \in \bar P, i \neq i^*} \left(b-s + \frac{(b-s) a_{i'1}}{a_{i^* j^*} + b - s} \right) x^*_{ij} \\
+ & a_{i^* j^*} x^*_{i^* j^*}  +   \sum_{j \in N_{i^*} - j^*} a_{i^* j^*} \max \left\{1, \frac{a_{i^* j} }{ a_{i^* j^*} + b - s} \right\} x^*_{i^* j} \\
> & b + (|\bar P| - 2) (b-s) \pare{1 +  \frac{ a_{i'1}}{a_{i^* j^*} + b - s}} \\
= & s - a_{i' 1} +  \frac{a_{i^* j^*} a_{i'1}}{a_{i^* j^*} + b - s} + \left(b-s + \frac{(b-s) a_{i'1}}{a_{i^* j^*} + b - s} \right) (|\bar P|  - 1).
\end{split}
\end{equation*}
We obtain
\begin{equation*}
\begin{split}
& \left(b-s + \frac{(b-s) a_{i'1}}{a_{i^* j^*} + b - s} \right) \sum_{ij \in \bar P, i \neq i^*}  (1-x^*_{ij})
< \frac{a_{i^* j^*} \cdot a_{i' 1}}{a_{i^* j^*} + b - s} x^*_{i' 1} \\ + & \sum_{i \in M_P - \{i', i^*\}} \sum_{j \in N_i} a_{ij} x^*_{ij} + a_{i^* j^*} x^*_{i^* j^*}  + \sum_{j \in N_{i^*} - j^*} a_{i^* j^*} \max \left\{1, \frac{a_{i^* j} }{ a_{i^* j^*} + b - s} \right\} x^*_{i^* j} \\
& - s + a_{i' 1} -  \frac{a_{i^* j^*} a_{i'1}}{a_{i^* j^*} + b - s}\\
\leq & a_{i' 1} x^*_{i' 1} + \sum_{i \in M_P - \{i', i^*\}} \sum_{j \in N_i} a_{ij} x^*_{ij} + \sum_{j \in N_{i^*}} a_{i^* j} x^*_{i^* j} - s + a_{i' 1} -  \frac{a_{i^* j^*} a_{i'1}}{a_{i^* j^*} + b - s} \\
\leq & b - s + a_{i' 1} -  \frac{a_{i^* j^*} a_{i'1}}{a_{i^* j^*} + b - s} = b-s + \frac{(b-s) a_{i'1}}{a_{i^* j^*} + b - s}.
\end{split}
\end{equation*}
Therefore, we have $\sum_{ij \in \bar P, i \neq i^*}  (1-x^*_{ij}) < 1$, and $|\{ij \in \bar P \mid i \neq i^*, j \neq o_{i,1}\}| \leq 1$ follows naturally. 
\end{proof}

% Moreover, based on the proof of Theorem~\ref{theo: comp_ineq_2} and Theorem~\ref{theo: comp_ineq_3}, in order for the pack $P$ to produce separating complementarity pack inequalities, $P$ has to be a subset in $\supp(x^*)$.
% \begin{lemma}
%     If $P$ is a pack and the corresponding inequality \eqref{ieq: 3} or \eqref{ieq: 4} separates $x^*$, then $P \subseteq \supp(x^*)$.
% \end{lemma}

For the cuts \eqref{ieq: 3} and \eqref{ieq: 4}, we give the exact separation algorithms in Algorithm~\ref{alg: 2} and Algorithm~\ref{alg: 3} below. 
The validity of these algorithms follows from Lemma~\ref{lem: sep_heu2}, and we omit the proofs here. 

\begin{algorithm}[H]
   \caption{Exact separation of \eqref{ieq: 3}}
   \label{alg: 2}
       \hspace*{\algorithmicindent} \textbf{Input:} 
       An instance of \eqref{CKP} and an infeasible solution $x^*$. \\
       \hspace*{\algorithmicindent} \textbf{Output:} 
       A separating inequality in the form of \eqref{ieq: 3} if one exists.
    \begin{algorithmic}[1]
    \State Let $x^*_{i,o_{i,1}} \geq x^*_{i,o_{i,2}} \geq \ldots \geq x^*_{i, o_{i,n_i}}$ for any $i \in M-M_0$, here $o_{i,1}, \ldots, o_{i,n_i}$ are different indices in $ [n_i]$. 
    \For{$i^* \in M-M_0$}
    \For {$\bar M \subseteq M-M_0-i^*$}
    \If {$\sum_{i \in \bar M} (1-x^*_{i,o_{i,1}})\geq 1$}
    \State \textbf{Continue}
    \Else
        \For{$i' \in \bar M$}
            \For{$j' = o_{i', 1}, o_{i', 2}, \ldots, o_{i', n_{i'}}$}
            \State Set $\bar P := i'j' \cup \{(i,o_{i,1}) \mid \forall i \in \bar M-i'\}$. 
            \If {$\sum_{ij \in \bar P} (1-x^*_{ij})\geq 1$}
    \State \textbf{Break}
    \ElsIf{$\bar P \cup i^*n_{i^*}$ is not a pack}
    \State \textbf{Continue}
    \Else
    \For{$P_0 \subseteq \{i0 \mid  i \in M_0\}$}
    \State Set $P:= \bar P \cup P_0 \cup i^*n_{i^*}$.
    \If{$P$ is a pack and $x^*$ violates \eqref{ieq: 3} given by $P$ and $i^*$}
    \State \Return: A separating inequality \eqref{ieq: 3}.
    \EndIf
    \EndFor
    \EndIf
            \EndFor
        \EndFor
    \EndIf
    \EndFor
    \EndFor
\end{algorithmic}
\end{algorithm}

\begin{algorithm}[H]
   \caption{Exact separation of \eqref{ieq: 4}}
   \label{alg: 3}
       \hspace*{\algorithmicindent} \textbf{Input:} 
       An instance of \eqref{CKP} and an infeasible solution $x^*$. \\
       \hspace*{\algorithmicindent} \textbf{Output:} 
       A separating inequality in the form of \eqref{ieq: 4} if one exists.
    \begin{algorithmic}[1]
    \State Let $x^*_{i,o_{i,1}} \geq x^*_{i,o_{i,2}} \geq \ldots \geq x^*_{i, o_{i,n_i}}$ for any $i \in M-M_0$, here $o_{i,1}, \ldots, o_{i,n_i}$ are different indices in $ [n_i]$. 
    \For{$i^* \in M-M_0$}
    \For {$\bar M \subseteq M-M_0-i^*$}
    \If {$\sum_{i \in \bar M} (1-x^*_{i,o_{i,1}})\geq 1$}
    \State \textbf{Continue}
    \Else
        \For{$i' \in \bar M$}
            \For{$j' = o_{i', 1}, o_{i', 2}, \ldots, o_{i', n_{i'}}$}
            \State Set $\bar P := i'j' \cup \{(i,o_{i,1}) \mid \forall i \in \bar M-i'\}$. 
            \If {$\sum_{ij \in \bar P} (1-x^*_{ij})\geq 1$}
    \State \textbf{Break}
    \ElsIf{$\bar P \cup i^*n_{i^*}$ is not a pack}
    \State \textbf{Continue}
    \Else
    \For{$P_0 \subseteq \{i0 \mid  i \in M_0\}$}
    \State Set $P:= \bar P \cup P_0 \cup i^*n_{i^*}$.
    \For{$i' \in M_{P_0}$}
    \If{$P$ is a pack and $x^*$ violates \eqref{ieq: 4} given by $P, i^*$ and $i'$}
    \State \Return: A separating inequality \eqref{ieq: 4}.
    \EndIf
    \EndFor
    \EndFor
    \EndIf
            \EndFor
        \EndFor
    \EndIf
    \EndFor
    \EndFor
\end{algorithmic}
\end{algorithm}

\section{Future directions}
\label{sec: future}

In order for the proposed inequalities to be of practical use, it is necessary to develop more efficient separation heuristics. One interesting question for future work is: Can we leverage the intuition of our exact separation algorithms in Sect.~\ref{sec: heuristics} to devise polynomial-time heuristic algorithms? 
To examine the efficacy of those separation methods, an extensive numerical study should be performed. 
Another interesting direction of research spawns from the observation that the embedding conflict graph of the complementarity constraints in \eqref{CKP} is given by the union of some non-overlapping cliques.
An interesting question is then, whether we can gain a deep understanding of the valid inequalities for the following set, with respect to a general graph $G = ([n], E)$:
$$
\conv \left( \left\{x \in [0,1]^n \mid a^T x \leq b, x_i \cdot x_j = 0 \ \forall \{i,j\} \in E\right\} \right).
$$

\bigskip
\noindent
\textbf{Acknowledgments:}
The authors dedicate this paper to the memory of their good colleague Ismael Regis de Farias, whose life and work continues to provide inspiration.
A. Del Pia is partially funded by ONR grant N00014-19-1-2322. Any opinions,
findings, and conclusions or recommendations expressed in this material are those of the authors and do not necessarily reflect the views of the Office of Naval Research.

\bibliographystyle{plain}
\bibliography{biblio}

% \appendix

% \begin{appendices}

% \section{}
% \label{sec: appendix}

% \end{appendices}

\end{document}